\documentclass[12pt,english]{article}
\usepackage[a4paper]{geometry}
\usepackage{setspace}

\RequirePackage{amsthm,amsmath,amsfonts,amssymb}
\RequirePackage[authoryear, longnamesfirst]{natbib}
\RequirePackage{color}
\RequirePackage{mathtools}
\RequirePackage{caption}

\numberwithin{equation}{section}

\theoremstyle{plain}
\newtheorem{theorem}{Theorem}[section]
\newtheorem{lemma}{Lemma}[section]
\newtheorem{proposition}{Proposition}[section]
\newtheorem{corollary}{Corollary}[section]

\theoremstyle{definition} 
\newtheorem{example}{Example}[section]
\newtheorem{remark}{Remark}[section]
\newtheorem{definition}{Definition}[section]

\newtheorem{assumption}{Assumption}[section]

\newenvironment{mylist}[1]
	{\begin{list}{}
		{\settowidth{\labelwidth}{#1}
		 \setlength{\leftmargin}{\labelwidth}
		 \addtolength{\leftmargin}{\labelsep}
		 }}
	{\end{list}}

\newcommand\tsum{\textstyle\sum\nolimits}

\newcommand{\cZ}{{\mathfrak Z}}

\newcommand{\cC}{{\mathfrak C}}
\newcommand{\cN}{{\mathfrak N}}

\newcommand{\N}{{\cal N}}

\newcommand{\T}{{\cal T}}

\newcommand{\var}{{\rm Var}}
\newcommand{\cov}{{\rm Cov}}

\newcommand{\V}{{\cal V}}

\newcommand{\F}{{\cal F}}

\newcommand{\C}{{\cal C}}

\newcommand{\I}{{\cal I}}

\newcommand{\X}{{\cal X}}

\newcommand{\LL}{{\cal L}}

\newcommand{\half}{ \mbox{\small$\frac{1}{2}$}}

\newcommand{\be}{\begin{equation}}
\newcommand{\ee}{\end{equation}}

\newcommand{\bgamma}{\mbox{\boldmath$\gamma$}}

\def\w{\omega}
\def\e{\varepsilon}

\def\vv{\vartheta}

\def\dst {{\rightsquigarrow}}

\def\bbr{{\Bbb{R}}} 
\def\bbe{{\Bbb{E}}} 

\begin{document}
\title{\vspace{20pt}
Minimax asymptotics\,%
\thanks{%
The first author thanks
the Research Council of Finland,
Foundation for the Advancement of Finnish Securities Markets,
and OP Group Research Foundation
for financial support.
The second author thanks
Air Force Office of Scientific Research (AFOSR),
  Grant FA9550-22-1-0244.
Contact addresses:
Mika Meitz,
Department of Economics,
University of Helsinki,
P. O. Box 17, FI\textendash 00014 University of Helsinki, Finland;
e-mail: mika.meitz@helsinki.fi.
Alexander Shapiro,
School of Industrial and Systems Engineering,
Georgia Institute of Technology,
Atlanta, Georgia 30332, United States;
e-mail: ashapiro@isye.gatech.edu.%
}\vspace{20pt}
}
\author{%
Mika Meitz\\\small{University of Helsinki}
\and
Alexander Shapiro\\\small{Georgia Institute of Technology}
\vspace{20pt}
}
\date{April 2025}
\maketitle
\begin{abstract}
\noindent%
In this paper, we consider asymptotics
of the optimal value and the optimal solutions of parametric minimax estimation problems.
Specifically, we consider estimators of the optimal value and the optimal solutions in a sample minimax problem that approximates the true population problem and study the limiting distributions of these estimators as the sample size tends to infinity.
The main technical tool we employ in our analysis is the theory of sensitivity analysis of parameterized mathematical optimization problems.
Our results go well beyond the existing literature and show that these limiting distributions are highly non-Gaussian in general and normal in simple specific cases.
These results open up the way for the development of statistical inference methods in parametric minimax problems.

\bigskip{}
\bigskip{}
\bigskip{}

\noindent\textbf{MSC2020 subject classifications:} 
Primary 62F12; secondary 90C47.

\bigskip{}


\end{abstract}
\vfill{}

\pagebreak{}


\section{Introduction}

There exists a vast literature on general minimax problems
originating nearly a century ago with the seminal contribution of
\citet{vonneumann1928theorie}.
Early books on different aspects of minimax problems include
\citet{vonneumann1944theory},
\citet{danskin1967theory},
and
\citet{demyanov1974introduction}.
Modern discussions of minimax problems can be found, for instance, in the books of
\citet{myerson1991game} on game theory,
\citet{bertsekas2003convex} on optimization theory,
and \citet{SDR} on stochastic programming.
More recently, there has been a surge of interest in minimax problems in computer science and machine learning; see, for instance,
the Nevanlinna Prize lecture of
\citet{daskalakis2018equilibria}
or the articles of
\citet{jin2020what},
\citet{diakonikolas2021efficient}, and
\citet{Meitz2024SJS}
for further references and discussion.

In this paper, we consider the canonical formulation of a stochastic minimax problem given by
\begin{equation}\label{intro-minimax-pop}
 \min_{\gamma\in \Gamma}
 \biggl\{ \sup_{\xi\in \Xi} f(\gamma,\xi) \biggr\}
\quad\text{with}\quad
f(\gamma,\xi):=\bbe[F(X,\gamma,\xi)],
\end{equation}
where $\Gamma\subset \bbr^n$ and $\Xi\subset \bbr^m$ are (nonempty) {\em compact } sets, $X\in \X\subset  \bbr^d$ is a random vector, and $F:  \X\times\Gamma\times  \Xi\to \bbr$ is a continuous function.
This formulation is in line with the references cited above.
Given an IID random sample $X_1,\ldots
,X_N$ of $X$, the empirical counterpart of problem \eqref{intro-minimax-pop} is the sample minimax problem
\begin{equation}\label{intro-minimax-sam}
 \min_{\gamma\in \Gamma} \biggl\{ \sup_{\xi\in \Xi} \hat{f}_N(\gamma,\xi)\biggr\}
\quad\text{with}\quad
\hat{f}_N(\gamma,\xi):=N^{-1}\sum_{j=1}^N F(X_j,\gamma,\xi).
\end{equation}
We denote the optimal values of problems \eqref{intro-minimax-pop} and \eqref{intro-minimax-sam} by
\[
\vv^*:=\inf_{\gamma\in \Gamma} \sup_{\xi\in \Xi} f(\gamma,\xi)
\qquad\text{and}\qquad
\hat{\vv}_N:=\inf_{\gamma\in \Gamma} \sup_{\xi\in \Xi} \hat{f}_N(\gamma,\xi),
\]
respectively.
Furthermore, the respective sets of optimal $\gamma$- and $\xi$-solutions to the population and sample minimax problems are denoted by
\begin{alignat}{2}
\notag
\Gamma^* &:=& \arg\min_{\gamma\in \Gamma} \sup_{\xi\in \Xi} f(\gamma,\xi),
\qquad
\Xi^*(\gamma)&:=\arg\max_{\xi\in \Xi}f(\gamma,\xi), \\
\notag
\widehat{\Gamma}_N &:=& \arg\min_{\gamma\in \Gamma} \sup_{\xi\in \Xi} \hat{f}_N(\gamma,\xi),
\qquad
\widehat{\Xi}_N(\gamma)&:=\arg\max_{\xi\in \Xi} \hat{f}_N (\gamma,\xi).
\end{alignat}
The aim of this article is to study the asymptotic behavior,
in particular the limiting distribution,
of the optimal value $\hat{\vv}_N$ and the optimal $\gamma$-solutions, say $\hat{\gamma}_N \in \widehat{\Gamma}_N$, as the sample size $N$ tends to infinity.

Regarding previous literature,
consistency results for optimal values and optimal solutions are well-known.
These results are documented for instance in
\citet{SDR} where it is shown that,
with probability one,
$\hat{\vv}_N$ tends to $\vv^*$ and  the distance from  $\hat{\gamma}_N$ to $\Gamma^*$ tends to zero for any
$\hat{\gamma}_N \in \widehat{\Gamma}_N$.
For details of these consistency results, see \citet[Theorem 5.9]{SDR},
and for slight refinements, see \citet[Theorem 1]{Meitz2024SJS}.

Much less and surprisingly little is known about the
limiting distributions of $\hat{\vv}_N$ and $\hat{\gamma}_N$,   and the present paper aims to address this gap in the literature.
Asymptotics of the optimal value $\hat{\vv}_N$
have, to the best of our knowledge, been discussed in one paper,
namely \citet{sha08}.
He considered a convex-concave special case in which the sets $\Gamma$ and $\Xi$ are convex and
$f(\cdot,\xi)$ is convex for all fixed $\xi\in\Xi$ and $f(\gamma,\cdot)$ is concave for all fixed $\gamma\in\Gamma$,
and in his Theorem 3.1 characterized the limiting distribution of
$\hat{\vv}_N$. This result is also reported in \citet[Theorem 5.10]{SDR}.
As for the limiting distribution of optimal solutions, as far as we know, only two papers exist that in any way address this issue.
\cite{Biau2020AS} consider a specific formulation of the minimax problem corresponding to so-called generative adversarial networks.
Assuming that a single unique solution to their specific minimax problem exists,
their Theorem 4.3 gives a result on asymptotic normality of
the optimal $\gamma$-solutions.
\cite{Meitz2024SJS} does not consider asymptotics of optimal solutions per se, but in his Theorem 2 suggests a subsampling-based procedure for forming confidence sets for the optimal solutions.
It seems that for general minimax problems,
results regarding the asymptotics of optimal solutions have not yet appeared in the literature.

In this paper, we go well beyond these existing results in
characterizing the limiting distributions of the optimal values $\hat{\vv}_N$ and the optimal solutions $\hat{\gamma}_N$.
The main technical tool of our analysis is the modern theory of sensitivity
analysis of parameterized mathematical optimization problems as laid out in \citet{BS}.
Early antecedents of this approach (applied, e.g., to minimization problems with equality and inequality constraints) include \citet{shap1989, shap1991, shap1993} and \citet{KR1993}.
Our use of this theory makes the present paper quite different from the
previous works of
\citet{sha08}, \citet{Biau2020AS}, and \citet{Meitz2024SJS}
who employed other techniques to obtain their results.

In our Section \ref{sec-value} we consider asymptotics of the optimal value $\hat{\vv}_N$. A first step in our analysis is to write the minimax problem \eqref{intro-minimax-pop} as an equivalent semi-infinite mathematical programming problem. For such problems, a general theory of sensitivity analysis is provided in Section 5.4.4 of \citet{BS}. Making use of these results, first-order properties of the minimax problem, and the functional delta method we are able to express the limiting distribution of the optimal value in two scenarios that are not convex-concave and that complement the earlier results obtained by \citet{sha08} in the convex-concave case.
No restrictions are imposed on the number of $\gamma$- and $\xi$-solutions.

Section \ref{sec-sol} considers the limiting distribution of optimal solutions $\hat{\gamma}_N$. This is a more delicate issue and relies on second-order properties of the minimax problem.
We assume the $\gamma$-parameter is point identified, i.e. $\Gamma^* = \{\gamma^*\}$, but allow for a finite number of $\xi$-solutions.
This allows us to apply the so-called reduction method used in the second-order analysis of semi-infinite programs (cf., \citealp{HK1993})
and reformulate the original minimax problem as a nonlinear programming problem with a finite number of inequality constraints.
Careful analysis of this problem leads to our main result, Theorem \ref{th-gamma}, characterizing the limiting distribution of $\hat{\gamma}_N$.

Our paper takes a definitive step forward in characterizing the limiting distributions of optimal values $\hat{\vv}_N$ and optimal solutions $\hat{\gamma}_N$ in parametric minimax estimation problems.
These limiting distributions are highly non-Gaussian in general and reduce to the normal distribution only when restrictive simplifying assumptions are imposed.
Our results also open up the possibility of performing statistical inference in minimax estimation problems. In particular, evaluation of approximate errors of these estimators, construction of confidence intervals, and forming appropriate critical regions for hypothesis tests would be of interest in practical applications.

It may be clarifying to note that minimax-type formulations have appeared in the literature also in other contexts, for instance in
(generalized) empirical likelihood estimation
(see, e.g., \citealp{owen2001} or \citealp{neweysmith2004})
and in distributionally robust optimization
(see, e.g., \citealp{delageye2010}).
These minimax problems are somewhat different from the one considered in this paper.

We use the following terminology and notations through the paper. We denote by $C(\Gamma \times \Xi)$ the space of continuous functions $\phi:\Gamma\times \Xi \to \bbr$ equipped with the sup-norm $\|\phi\|=\sup_{(\gamma,\xi)\in \Gamma\times \Xi}|\phi(\gamma,\xi)|$ (recall that the sets $\Gamma$ and $\Xi$ are assumed to be compact).
The symbol
$\dst$ denotes weak convergence and
$\delta_\xi$ the Dirac measure of mass one at $\xi$.
Section 2 considers asymptotics of the optimal value and relevant notions of directional differentiability will be defined in Section \ref{sec-dirdiff}.
Section 3 discusses asymptotics of the optimal solutions, and concluding remarks will be given in Section 4.

\setcounter{equation}{0}
\section{Asymptotics of the optimal value}
\label{sec-value}

Throughout the paper we consider the population and sample minimax problems \eqref{intro-minimax-pop} and \eqref{intro-minimax-sam}. We make the following basic assumption.

\begin{assumption} \label{ass-1}
There is a measurable function $\psi(X)$ such that $\bbe[\psi(X)^2]$ is finite and
\[
| F(X,\gamma,\xi)-F(X,\gamma',\xi') | 
\le \psi(X)(\|\gamma-\gamma'\|+\|\xi-\xi'\|)
\]
for all $\gamma,\gamma' \in \Gamma$, $\xi,\xi' \in \Xi$, and almost every $X$.
\end{assumption}

Under this assumption, the expected value function $f:\Gamma\times\Xi\to \bbr$ is Lipschitz continuous and the sets
$\Gamma^*$ and $\Xi^*(\gamma)$, $\gamma\in \Gamma$,
are nonempty  (recall that the sets $\Gamma$ and $\Xi$ are assumed to be compact).  Moreover,
$N^{1/2}(\hat{f}_N-f)$ converges in distribution to a mean zero Gaussian random element $\F(\gamma,\xi)$  of the space  $C(\Gamma\times \Xi)$ with covariance function
\begin{equation}\label{covar}
 \cov[\F(\gamma,\xi),\F(\gamma',\xi')] =
\cov[F(X,\gamma,\xi),F(X,\gamma',\xi')]
\end{equation}
(see, e.g., \citealp[Example 19.7]{vaart}).
This is denoted $N^{1/2}(\hat{f}_N-f)\dst  \F$.

We study asymptotics of the optimal value by 
first analyzing appropriate directional differentiability properties in Section \ref{sec-dirdiff}, 
and then convert these results into asymptotics in Section \ref{optval-asymptotics}.

\subsection{Directional differentiability of the optimal value function}
\label{sec-dirdiff}

To obtain our results, we next consider a minimax problem more general than the ones in
\eqref{intro-minimax-pop} and \eqref{intro-minimax-sam}.
Let the function
$\phi \in C ( \Gamma \times \Xi )$
be arbitrary, and consider the general optimal value function
\begin{equation}\label{optfunc}
V(\phi):=\inf_{\gamma\in \Gamma} \sup_{\xi\in \Xi} \phi(\gamma,\xi),
\end{equation}
depending on the function $\phi$.
Of course, $V(f) = \vv^*$ and $V(\hat{f}_N) = \hat{\vv}_N$.
Moreover, under the above assumptions, the function $V:C(\Gamma\times \Xi) \to \bbr$ is finite-valued and Lipschitz continuous (in the norm topology of the space $C ( \Gamma \times \Xi )$).
In what follows, we are interested in differentiability properties of the value function $V(\phi)$ in the vicinity of the point $\phi_0 = f$.
 It is said that $V(\cdot)$  is G\^ateaux directionally differentiable at $\phi \in C(\Gamma\times \Xi)$ if the directional derivative
\[
V'_\phi(\eta):=\lim_{t\downarrow 0}\frac{V(\phi+t\eta)-V(\phi)}{t}
\]
exists for every direction $\eta\in C(\Gamma\times \Xi)$.
Moreover, it is said that $V(\cdot)$  is Hadamard  directionally differentiable if
\[
V'_\phi(\eta)=\lim_{t\downarrow 0,\, \eta'\to \eta}\frac{V(\phi+t\eta')-V(\phi)}{t},
\]
where the convergence $\eta'\to \eta$ is understood in the norm topology of $C(\Gamma\times \Xi)$.

For a   set $K\subset  C(\Gamma\times \Xi) $, it is said that $V(\cdot)$  is Hadamard  directionally differentiable at $\phi$  {\em tangentially} to the set $K$  if the limit
\[
V'_\phi(\eta)=\lim_{t\downarrow 0,\,  \eta'\to_{_K} \eta}\frac{V(\phi+t\eta')-V(\phi)}{t}
\]
exists for any $\eta\in \T_K(\phi)$. Here  $\T_K(\phi)$   denotes the contingent
cone to $K$ at $\phi\in K$, and $ \eta'\to_{_K} \eta$ means that  $\eta'\to  \eta$ and $\phi+t\eta'\in K$. We   will be interested in the case where the set $K$ is convex and closed. In that case $\T_K(\phi)$ becomes the standard tangent cone to $K$ at $\phi$, given by the topological closure of the radial cone consisting of  $\eta$ such that  $\phi+t\eta\in K$ for all $t>0$ small enough.
Also note that since $V:C(\Gamma\times \Xi)\to \bbr$ is Lipschitz continuous, it follows that if $V(\cdot)$  is G\^ateaux directionally differentiable, then it is Hadamard  directionally differentiable and $V'_\phi(\cdot)$ is locally Lipschitz continuous.

We begin with the following  directional differentiability result for the optimal  value function $V$.

\begin{proposition}
\label{hdd-1}
Suppose the following condition holds: {\rm (i)} for every $\gamma^*\in \Gamma^*$ and every $\xi^*\in \Xi^*(\gamma^*)$, the point $\gamma^*$ is a local minimizer of the function $f(\gamma,\xi^*)$, $\gamma\in \Gamma$. Then
the optimal value function  $V(\cdot)$ is Hadamard directionally differentiable at $f$ and
\begin{equation}\label{deriv-a}
V'_f(\eta)=\min_{\gamma\in \Gamma^*}\sup_{\xi\in \Xi^*(\gamma)} \eta(\gamma,\xi).
\end{equation}
\end{proposition}

\begin{proof}
First consider an arbitrary $\phi \in C(\Gamma\times \Xi)$,
and for each $\gamma \in \Gamma$ denote
$\Xi^*_{\phi}(\gamma) := \arg\max_{\xi\in\Xi} \phi(\gamma,\xi)$.
Consider the multifunction (point-to-set mapping)
$(\phi,\gamma) \mapsto  \Xi^*_{\phi}(\gamma)$
mapping $(\phi,\gamma) \in C(\Gamma \times\Xi) \times \Gamma$
to the set $2^{\Xi}$ of subsets of $\Xi$.
As $\phi$ are continuous and $\Xi$ is compact, this multifunction is closed, i.e.,
if $(\phi_k,\gamma_k) \to (\phi,\gamma)$, $\xi_k \in  \Xi^*_{\phi_k}(\gamma_k)$ and $\xi_k\to \xi$, then $\xi\in \Xi^*_{\phi}(\gamma)$
(see, e.g., the remarks following \citealp[Proposition 4.4]{BS}).
Moreover, this multifunction is upper semicontinuous, i.e., for any
$(\bar{\phi},\bar{\gamma}) \in C(\Gamma \times\Xi) \times \Gamma$
and a neighborhood $\N$ of $\Xi^*_{\bar{\phi}}(\bar{\gamma})$,
it follows that $\Xi^*_{\phi}(\gamma)\subset \N$ for any
$(\phi,\gamma)$ sufficiently close to $(\bar{\phi},\bar{\gamma})$
(see, e.g., \citealp[Lemma 4.3]{BS}).

Now let $\eta\in C(\Gamma\times \Xi)$ and $t_k\downarrow 0$.
Consider $\gamma^*\in \Gamma^*$ and
$\xi_k\in \arg\max_{\xi\in\Xi}f(\gamma^*,\xi)+t_k\eta(\gamma^*,\xi)$.
We have
\[
V(f+t_k\eta)\le \sup_{\xi\in \Xi}\{f(\gamma^*,\xi)+t_k \eta(\gamma^*,\xi)\}=f(\gamma^*,\xi_k)+t_k \eta(\gamma^*,\xi_k).
\]
As the multifunction $(\phi,\gamma)\mapsto  \Xi^*_{\phi}(\gamma)$ is upper semicontinuous,
the distance from $\xi_k$ to the set $\Xi^*_f(\gamma^*)=\Xi^*(\gamma^*)$ tends to zero. Thus  by passing to a subsequence if necessary  we can assume that $\xi_k$ converges to a point $\xi^*\in \Xi^*(\gamma^*)$. Then
\[
V(f+t_k\eta)\le f(\gamma^*,\xi^*)+t_k \eta(\gamma^*,\xi^*)+o(t_k).
\]
As $V(f)= f(\gamma^*,\xi^*)$,
it follows that
\[
   \limsup_{t\downarrow 0}\frac{V(f+t\eta)-V(f)}{t}
   \le \sup_{\xi\in \Xi^*(\gamma^*)} \eta(\gamma^*,\xi).
\]
Since this holds for any $\gamma^*\in \Gamma^*$ it follows that
 \begin{equation}\label{limsup}
  \limsup_{t\downarrow 0}\frac{V(f+t\eta)-V(f)}{t}
  \le \inf_{\gamma\in \Gamma^*} \sup_{\xi\in \Xi^*(\gamma)} \eta(\gamma,\xi).
 \end{equation}

 Now consider sequences
 \begin{eqnarray}
 \label{sec-1}
 \gamma_k&\in& \arg\min_{\gamma\in \Gamma}\sup_{\xi\in \Xi}\{ f(\gamma,\xi)+t_k\eta(\gamma,\xi)\},\\
 \xi_k&\in &
 \label{sec-2}
 \arg\max_{\xi\in\Xi}\{f(\gamma_k,\xi)+t_k\eta(\gamma_k,\xi)\}.
 \end{eqnarray}
By \eqref{sec-1} and \eqref{sec-2} we  have that
\[
V(f+t_k\eta)=f(\gamma_k,\xi_k)+t_k\eta(\gamma_k,\xi_k).
\]
 By passing to a subsequence if necessary we can assume that
 $\gamma_k$ converges to a point $\gamma^*\in \Gamma^*$
 and $\xi_k$ converges to a point $\xi^*\in \Xi$.
As the multifunction $(\phi,\gamma)\mapsto  \Xi^*_{\phi}(\gamma)$ is upper semicontinuous,
it follows that $\xi^*\in \Xi^*(\gamma^*)$, and hence
    $V(f)=f(\gamma^*,\xi^*)$.

By condition (i) we have that $\gamma^*$ is a local minimizer of $f(\gamma,\xi^*)$, $\gamma\in \Gamma$.
It follows  that for $k$ large enough,
  $f(\gamma_k,\xi^*)\ge f(\gamma^*,\xi^*)$, and
 thus
 \[
 V(f+t_k\eta)
 \ge f(\gamma_k,\xi^*)+t_k\eta(\gamma_k,\xi^*)
 \ge f(\gamma^*,\xi^*)+t_k\eta(\gamma_k,\xi^*).
 \]
   It follows that
 \begin{equation}\label{liminf}
  \liminf_{t\downarrow 0}\frac{V(f+t\eta)-V(f)}{t}\ge \eta(\gamma^*,\xi^*).
 \end{equation}
As the arguments above hold for an arbitrary choice of the sequence $\xi_k$ and an arbitrary limit point $\xi^* \in \Xi^*(\gamma^*)$, the lower bound in \eqref{liminf} can be replaced with $\sup_{\xi\in \Xi^*(\gamma^*)} \eta(\gamma^*,\xi)$,
and thus also with $\inf_{\gamma\in \Gamma^*} \sup_{\xi\in \Xi^*(\gamma)} \eta(\gamma,\xi)$.
Together with \eqref{limsup} this completes the proof.
\end{proof}

It could be noted that the {\em  upper} bound \eqref{limsup} does not involve assumption (i) of Proposition \ref{hdd-1}. The following example shows that the upper bound \eqref{limsup}  can be strict without an additional condition, such as assumption (i),  even if problem \eqref{intro-minimax-pop} has a unique optimal solution $\gamma^*$, i.e.,  $\Gamma^*=\{\gamma^*\}$ is the singleton.

\begin{example}
Let $\Gamma=[-1,1]\subset \bbr$, the set $\Xi=\{\xi_1,\xi_2\}$ consists of two points, and \linebreak $f(\gamma,\xi_1)=-\gamma$, $f(\gamma,\xi_2)=\gamma$. Then
$V(f)=\min_{-1\le \gamma\le 1} \max\{-\gamma,\gamma\}=0$, $\Gamma^*=\{\gamma^*\}$, $\gamma^*=0$,   and $\Xi^*(\gamma^*)=\{\xi_1,\xi_2\}$. Consider $\eta(\gamma,\xi)$ with $\eta(\gamma,\xi_1)\equiv 1$ and $\eta(\gamma,\xi_2)\equiv 0$. Then
\[
V(f+t\eta)=\min_{-1\le \gamma\le 1} \max\{-\gamma+t,\gamma\}=t/2,
\]
and hence $V'_f(\eta)=1/2$. On the other hand,
\[
\min_{\gamma\in \Gamma^*}\sup_{\xi\in \Xi^*(\gamma)} \eta(\gamma,\xi)=\max\{\eta(\gamma^*,\xi_1), \eta(\gamma^*,\xi_2)\}=1.
\]
That is, the inequality \eqref{limsup}   is strict and \eqref{deriv-a} does not hold.
Instead, equation \eqref{dirderiv} (below)  gives the correct formula for the directional derivative.
\end{example}

Proposition \ref{hdd-1} is a slight generalization of the corresponding result in \citet[Proposition 2.1]{sha08}, who considered the following convex-concave case. Suppose that the sets $\Gamma$ and $\Xi$ are convex, and $f(\gamma,\xi)$  is convex in $\gamma$ and concave in $\xi$. Then $\Xi^*=\Xi^*(\gamma^*)$ does not depend on   $\gamma^*\in \Gamma^*$  and,
for $\xi^*\in \Xi^*$, the point $(\gamma^*,\xi^*)$ is a saddle point of the minimax problem \eqref{intro-minimax-pop}. Consequently $\gamma^*$ is a minimizer of $f(\gamma,\xi^*)$  and hence condition {\rm (i)} of Proposition \ref{hdd-1} holds.
Therefore
\begin{equation}\label{dird-1}
 V'_f(\eta)= \inf_{\gamma^*\in \Gamma^*} \sup_{\xi^*\in \Xi^*} \eta(\gamma^*,\xi^*),
\end{equation}
which coincides with the respective formula (2.5) in \cite{sha08}.

The directional differentiability result of Proposition \ref{hdd-1} provides a connection to the earlier results of \cite{sha08}, but as condition (i) of this proposition is difficult to verify in practice, we next consider two other more concrete cases in which directional differentiability results can be established.

Towards this end,
first note that the population minimax problem \eqref{intro-minimax-pop} can be equivalently formulated as the semi-infinite programming problem
\begin{equation}\label{semi-f}
\min_{z\in \bbr, \gamma\in \Gamma} z
\quad\text{s.t.}\quad
f(\gamma,\xi)-z\le 0,
\;\forall \xi\in \Xi,
\end{equation}
where the minimization is over the variables $(z,\gamma) \in \bbr\times\Gamma\subset \bbr^{n+1}$.
Now $\vartheta^*$ and $(\gamma^*,\xi^*)$ are an optimal value and solution of minimax problem \eqref{intro-minimax-pop} if and only if $(z^*,\gamma^*)$ with $z^* = \vartheta^*$ is a solution of the problem \eqref{semi-f} with $\xi^* \in \arg\max_{\xi\in \Xi}f(\gamma^*,\xi)$.

More generally, let the function
$\phi \in C ( \Gamma \times \Xi )$
be arbitrary, and consider the general semi-infinite programming problem
\begin{equation}\label{semi-phi}
\min_{z\in \bbr, \gamma\in \Gamma} z
\quad\text{s.t.}\quad
\phi(\gamma,\xi)-z\le 0,
\;\forall \xi\in \Xi,
\end{equation}
which is equivalent to \eqref{optfunc}. When $\phi = f$, \eqref{semi-phi} becomes problem \eqref{semi-f}. When
$\phi = \hat{f}_N$, \eqref{semi-phi} becomes
\[
\min_{z\in \bbr, \gamma\in \Gamma} z
\quad\text{s.t.}\quad
\hat{f}_N(\gamma,\xi)-z\le 0,
\;\forall \xi\in \Xi,
\]
which is equivalent to the sample minimax problem \eqref{intro-minimax-sam}.
We continue denoting the optimal value of problem \eqref{semi-phi}, depending on the function $\phi\in C( \Gamma \times \Xi )$, with $V(\phi)$. Again, we are interested in differentiability properties of the value function $V(\phi)$ in the vicinity of the point $\phi_0 = f$.

To this end, we apply results from \citet[Sections 5.4 and 4.3]{BS}
in a differentiable case and make the following assumption.

\begin{assumption} \label{ass-1a}
\mbox{}\vspace*{-6pt}
\begin{mylist}{(iii)}
\item [{(i)}] For almost every $X$ and all $\xi\in \Xi$ the function $F(X,\gamma,\xi)$ is differentiable in $\gamma$
with $\nabla_\gamma F(X,\gamma,\xi)$   continuous in $(\gamma,\xi)$.
\item [{(ii)}]  There is an integrable function $C(X)$ such that  almost surely   $\|\nabla_\gamma F(X,\gamma,\xi)\|\le C(X)$
for all $(\gamma,\xi)\in\Gamma\times\Xi$.
\item [{(iii)}]  Every $\gamma\in \Gamma^*$ is an interior point of the set $\Gamma$.
\end{mylist}
\end{assumption}

It follows (e.g. \citealp[Theorem 7.44]{SDR})
that $f(\gamma,\xi)$ is continuously differentiable in $\gamma$ and
\[
 \nabla_\gamma f(\gamma,\xi)=\bbe[\nabla_\gamma F(X,\gamma,\xi)].
\]
Thus in problem \eqref{semi-f} the objective function and the inequality constraints are continuously differentiable in $(z,\gamma)$ for all fixed $\xi$ with the gradients with respect to $(z,\gamma)$  given by
\[
(1,0,\ldots,0) \qquad\text{and}\qquad (-1,\nabla_\gamma f(\gamma,\xi))
\]
and these derivatives are continuous in $(z,\gamma,\xi)$.

Next consider first-order optimality conditions of problem \eqref{semi-f}.
We follow \citet[Section 5.4.2]{BS}.
Suppose that $\gamma^*\in \Gamma^*$ and $z^*=\vv^*$ are an optimal solution and the optimal value of problem  \eqref{intro-minimax-pop} or, in other words, that $(z^*,\gamma^*)$ is a solution to \eqref{semi-f}, and that the corresponding optimal solutions in $\xi$ are $\Xi^*(\gamma^*)$.
 Note that the extended Mangasarian-Fromovitz  constraint qualification (e.g., \citealp[eq. (5.280)]{BS}) for the semi-infinite problem \eqref{semi-f} holds automatically.
It now follows from
first-order optimality conditions
(e.g., \citealp[Theorems 5.111 and 5.113]{BS})
that there exists a $k\le n+1$ and a discrete measure
$\mu=\sum_{i=1}^{k}\lambda_i \delta_{\xi_i}$
 on $\Xi$ such that
\begin{equation}\label{optcon}
\lambda_i\ge 0 \text{ and } \xi_i\in \Xi^*(\gamma^*) \text{ for } i=1,\ldots,k, \quad
\sum_{i=1}^{k}\lambda_i =1,
\quad\text{and}\quad
\sum_{i=1}^{k}\lambda_i \nabla_\gamma f(\gamma^*,\xi_i)=0.
\end{equation}
Denote by $\Lambda(\gamma^*)$  the set of discrete Lagrange multiplier measures
$\mu=\sum_{i=1}^{k}\lambda_i \delta_{\xi_i}$
satisfying \eqref{optcon} for $k = n+1$.
 Note that
by Carath{\'e}odory's theorem it suffices to take $k\le n+1$ in the optimality equations \eqref{optcon}.

We next consider the more general problem \eqref{semi-phi} and derive appropriate differentiability results for the value function $V(\phi)$ making use of results in \citet[Sec. 5.4.4]{BS}.
To this end, note that the Lagrangian of problem \eqref{semi-phi} is
\[
L(z,\gamma,\mu,\phi):=z+\int_\Xi (\phi(\gamma,\xi)-z)d\mu(\xi),
\]
where $\mu$ denotes any finite measure on $\Xi$.
In particular, if $\mu$ is any discrete measure of the form
$\mu=\sum_{i=1}^{k}\lambda_i \delta_{\xi_i}$, then
\[
L(z,\gamma,\mu,\phi):=z+\sum_{i=1}^{k}\lambda_i (\phi(\gamma,\xi_i)-z),
\]
and in this case the directional derivative of $L(z,\gamma,\mu,\phi)$ at $\phi=f$ in the direction
$\eta \in C( \Gamma \times \Xi )$ equals
\[
D_{\phi} L(z,\gamma,\mu,f) \eta = \sum_{i=1}^{k}\lambda_i \eta(\gamma,\xi_i).
\]

For any solution $(z^*,\gamma^*)$ of \eqref{semi-f},
the corresponding $\xi$-solutions $\Xi^*(\gamma^*)$,
and direction $\eta \in C(\Gamma \times \Xi)$,
an appropriate linearization of problem \eqref{semi-phi}
is given by
(cf. \citealp[eqn (5.322)]{BS})
\[
\min_{z\in \bbr, \gamma\in \Gamma} z
\quad\text{s.t.}\quad
\nabla_{\gamma} f (\gamma,\xi) + \eta(\gamma,\xi) - z \le 0,
\;\forall \xi\in \Xi^*(\gamma^*),
\]
and the corresponding Lagrangian dual problem takes the form
\begin{equation} \label{lag-dual}
\sup_{\mu\in \Lambda(\gamma^*)}
\sum_{i=1}^{n+1}\lambda_i \eta(\gamma^*,\xi_i).
\end{equation}

With these preparations, we can now state a result that gives an upper bound for a (potential) Hadamard directional derivative.
Recall  that $f\in C(\Gamma\times\Xi)$,  the   sets $\Gamma$ and $\Xi$ are assumed to be compact, and $\Lambda(\gamma^*)$ consists of
discrete measures
$\mu=\sum_{i=1}^{k}\lambda_i \delta_{\xi_i}$
 satisfying \eqref{optcon}.

\begin{proposition}
Suppose that Assumptions {\rm  \ref{ass-1}} and {\rm \ref{ass-1a}} hold. Then
for any $\eta\in C(\Gamma\times\Xi)$ it follows that
\begin{equation}\label{bound}
  \limsup_{t\downarrow 0}\frac{V(f+t\eta)-V(f)}{t}\le \inf_{\gamma^*\in \Gamma^*} \sup_{\mu\in \Lambda(\gamma^*)}\sum_{i=1}^{n+1}\lambda_i \eta(\gamma^*,\xi_i).
\end{equation}
\end{proposition}

\begin{proof}
We apply Proposition 5.121 of \cite{BS};
see also the remarks that follow it as well as Proposition 4.22.
Specifically, as the  extended Mangasarian-Fromovitz constraint qualification
holds,
so does the directional regularity condition (5.324) of \cite{BS}.
The upper bound \eqref{bound} now follows from Proposition 5.121 (and the remarks that follow it)
together with the form of the Lagrangian dual problem \eqref{lag-dual}.
\end{proof}

The bound on the right hand side of \eqref{bound} can be not tight, i.e. it can be strictly larger than the respective $\liminf_{t\downarrow 0}$, even if $\Gamma^*=\{\gamma^*\}$ is the singleton and $\Xi^*(\gamma^*)$ is finite
 (cf., \citealp[p. 523]{BS}).
However, there are two  important   cases in which the directional derivative  $V'_f(\eta)$ is equal to the right hand side of \eqref{bound}.

The first case occurs when $f(\gamma,\xi)$ is convex in $\gamma$
 and the set $\Gamma$ is convex.
In that case the set $\Lambda^*$ of Lagrange multipliers
coincides with the set of optimal solutions of the dual problem
\[
 \max_{\mu\succeq 0}\inf_{z\in \bbr, \gamma\in \Gamma}L(z,\gamma,\mu,f),
\]
where the maximization is over (nonnegative)  discrete measures on $\Xi$ (cf., \citealp[Section 5.4.1]{BS}).
Moreover,
the set $\Lambda^*$
 is nonempty, convex, bounded, and the same for any $\gamma^*\in \Gamma^*$
(cf., \citealp[Theorem 5.107]{BS};
note that the corresponding Slater condition holds here automatically).
Furthermore, the set $\Lambda^*$
  does not require  the differentiability conditions of Assumption \ref{ass-1a},
with the gradients being replaced by the respective subgradients
(cf., \citealp[Section 3.1.1, Definition 3.5]{BS}).
In this case we have the following result (cf., \citealp[Theorem 4.24]{BS}).
Consider the set $\cC$ of convex continuous functions $\nu:\Gamma\to \bbr$.
The set $\cC$  is a convex cone subset of  $C(\Gamma)$.
The corresponding set $$\bar{\cC}:=\big \{\psi\in C(\Gamma\times \Xi):\psi(\cdot,\xi)\in \cC,\;\xi\in \Xi\big\}$$ is a convex closed subset of  $C(\Gamma\times \Xi)$ consisting  of convex continuous   in $\gamma$ functions.

\begin{proposition}
\label{pr-dif1}
Suppose that Assumption {\rm \ref{ass-1}} holds, the set $\Gamma$ is convex  and    for
every $\xi\in \Xi$ the function $f(\gamma,\xi)$ is convex continuous in $\gamma$, i.e., $f\in\bar{\cC}$. Then the value function $V(\cdot)$ is Hadamard directionally differentiable at $f$ tangentially to the set   $\bar{\cC}$   and
\begin{equation}\label{dirderiv}
 V'_f(\eta)= \inf_{\gamma^*\in \Gamma^*} \sup_{\mu\in \Lambda^*}\sum_{i=1}^{n+1}\lambda_i \eta(\gamma^*,\xi_i).
\end{equation}
\end{proposition}

\begin{proof}
This result can be obtained from Theorem 4.24 of \cite{BS}
due to convexity and the validity of the directional regularity condition (5.324).
Condition (iii) of \citet[Theorem 4.24]{BS} (of continuity of optimal solutions) holds here by the compactness and continuity arguments (e.g., \citealp[Proposition 4.4]{BS} applied to the case where the feasible set does not depend on the parameters).
\end{proof}

\begin{remark}
If  Assumption \ref{ass-1a} also holds, 
we can again make a connection with 
the convex-concave case studied in \cite{sha08}.
When the sets $\Gamma$ and $\Xi$ are convex and
$f(\gamma,\xi)$  is convex in $\gamma$ and concave in $\xi$,
we noted above \eqref{dird-1} that for any
$\gamma^* \in \Gamma^*$ and $\xi^* \in \Xi^*$,
the point $\gamma^*$ is a minimizer of $f(\gamma,\xi^*)$
and consequently $\nabla_\gamma f(\gamma^*,\xi^*)=0$
(as it is now assumed that $\gamma^*$ is an interior point of $\Gamma$). It follows that the set $\Lambda (\gamma^*)$ consists of all possible Lagrange multipliers and hence \eqref{dirderiv} reduces to
\[
 V'_f(\eta)= \inf_{\gamma^*\in \Gamma^*} \sup_{\xi^*\in \Xi^*} \eta(\gamma^*,\xi^*).
\]
This coincides with the respective formula (2.5) in \cite{sha08}, where it was derived without assuming differentiability of $f(\cdot,\xi)$ and that the optimal solutions $\gamma^*$ are interior points of $\Gamma$.
\end{remark}

The second case where directional differentiability is obtained is when the set of Lagrange multipliers
$\Lambda(\gamma^*)=\{\mu(\gamma^*)\}$  is the singleton for every $\gamma^*\in \Gamma^*$.
Conditions for uniqueness of Lagrange multipliers are given in \citet[Theorem 5.114]{BS}. In particular, the following condition is sufficient for the uniqueness: For every $\gamma^*\in \Gamma^*$, the set $\Xi^*(\gamma^*)=\{\xi_1^*,\ldots,\xi_{k}^*\}$ is finite ($k$ may depend on $\gamma^*$) and the gradient vectors $\nabla_\gamma f(\gamma^*,\xi^*_i)$, $i=1,\ldots,k$, are affine independent (i.e., if $\sum_{i=1}^k a_i \nabla_\gamma f(\gamma^*,\xi^*_i)=0$ and $\sum_{i=1}^k a_i=0$, then $a_i=0$, $i=1,\ldots,k$). Note that this condition implies that $k\le n+1$.

In this second case, the previous results combined with
\citet[Theorem 4.26]{BS}
immediately give the following result.

\begin{proposition}
\label{pr-dif2}
Suppose that 
Assumptions {\rm \ref{ass-1}} and {\rm \ref{ass-1a}} hold,
and that $\Lambda(\gamma^*)=\{\mu(\gamma^*)\}$ is the singleton for every $\gamma^*\in \Gamma^*$,
where $\mu(\gamma^*)=\sum_{i=1}^{k}\lambda_i (\gamma^*)\delta_{\xi_i(\gamma^*)}$. Then the value function $V(\cdot)$ is Hadamard directionally differentiable at $f$ and
\[
 V'_f(\eta)= \inf_{\gamma^*\in \Gamma^*} \sum_{i=1}^{n+1}\lambda_i(\gamma^*) \eta(\gamma^*,\xi_i(\gamma^*)).
\]
\end{proposition}

\subsection{Asymptotics}
\label{optval-asymptotics}

The preceding results can be converted into asymptotics using the functional delta method.
For Hadamard directionally differentiable maps the functional
delta method is given in \citet[Thm 2.1]{shap1991}.
For convenience, we reproduce this result as the following lemma.

\begin{lemma} 
Suppose the map $V:C(\Gamma\times \Xi)\to \mathbb{R}$ is Hadamard directionally differentiable
at $\phi$ tangentially to a closed
convex set $K\subset C(\Gamma\times \Xi)$ with a directional  derivative $V_{\phi}^{\prime}:\T_K(\phi)\to \mathbb{R}$.
Let $\phi_{1},\phi_{2},\ldots,$ be a sequence of  $C(\Gamma\times \Xi)$-valued random
elements such that $\phi_i\in K$, $i=1,2,\ldots$, and
 $r_N(\phi_N-\phi)\rightsquigarrow \phi_{0}$
in $C(\Gamma\times \Xi)$ for some random element $\phi_{0}$
and  some constants $r_N \to \infty$. Then
\[
V(\phi_N) = V(\phi) + V_{\phi}^{\prime}( \phi_N-\phi ) + o_p(r_N^{-1})
\]
and
\[
r_N(V(\phi_N)-V(\phi)) \rightsquigarrow V_{\phi}^{\prime}(\phi_{0}).
\]
\end{lemma}

Recall that
$\vv^*=V(f)$ and $\hat{\vv}_N = V(\hat{f}_N)$, that $\Lambda(\gamma^*)$  denotes the  set of  Lagrange multiplier measures
$\mu=\sum_{i=1}^{k}\lambda_i \delta_{\xi_i}$
satisfying \eqref{optcon},   that $\F(\gamma,\xi)$ is a mean zero Gaussian random element with covariance function \eqref{covar} and that under Assumption \ref{ass-1},  $N^{1/2}(\hat{f}_N-f)\dst  \F$.
Proposition \ref{pr-dif1} together with the  functional delta theorem now imply the following asymptotics.

\begin{theorem}
Suppose that   Assumption  {\rm \ref{ass-1}}   holds, the set $\Gamma$ is convex and   for almost every $X$ and  every $\xi\in \Xi$ the function $F(X,\gamma,\xi)$ is convex in $\gamma$. Let $\Lambda^*$ be the set of Lagrange multipliers which  coincides with the set of optimal solutions of the dual problem.  Then
\[
 \hat{\vv}_N= \inf_{\gamma^*\in \Gamma^*} \sup_{\mu\in \Lambda^*}\sum_{i=1}^{n+1}\lambda_i \hat{f}_N(\gamma^*,\xi_i)+o_p(N^{-1/2}),
\]
and
\begin{equation}\label{asymval-2}
 N^{1/2}(\hat{\vv}_N-\vv^*)\dst \inf_{\gamma^*\in \Gamma^*} \sup_{\mu\in \Lambda^*}\sum_{i=1}^{n+1}\lambda_i \F(\gamma^*,\xi_i).
\end{equation}
\end{theorem}

Without the convexity condition,
by Proposition \ref{pr-dif2} and the  functional delta theorem we have the following asymptotics.

\begin{theorem}
Suppose that   Assumptions {\rm  \ref{ass-1}} and  {\rm \ref{ass-1a}}  hold, and  $\Lambda(\gamma^*)=\{\mu(\gamma^*)\}$    is the singleton for every $\gamma^*\in \Gamma^*$. Then
\[
 \hat{\vv}_N= \inf_{\gamma^*\in \Gamma^*}  \sum_{i=1}^{n+1} \lambda_i (\gamma^*)\hat{f}_N(\gamma^*,\xi_i(\gamma^*))+o_p(N^{-1/2}),
\]
and
\begin{equation}\label{asymval-2a}
 N^{1/2}(\hat{\vv}_N-\vv^*)\dst \inf_{\gamma^*\in \Gamma^*} \sum_{i=1}^{n+1}\lambda_i (\gamma^*) \F(\gamma^*,\xi_i(\gamma^*)).
\end{equation}
 \end{theorem}

The asymptotic distributions
\eqref{asymval-2} and \eqref{asymval-2a}
are typically non-Gaussian and rather complicated when there are multiple solutions in $\Gamma^*$   and/or multiple Lagrange multipliers.
However, when both
$\Gamma^*=\{\gamma^*\}$ and $\Lambda(\gamma^*)=\{\mu^*\}$ are singletons with
$\mu^* = \sum_{i=1}^{k} \lambda_i^* \delta_{\xi_i^*}$, we have the following.

\begin{corollary}
Suppose that   Assumptions {\rm  \ref{ass-1}} and  {\rm \ref{ass-1a}}  hold, and $\Gamma^*=\{\gamma^*\}$ and $\Lambda(\gamma^*)=\{\mu^*\}$ are singletons.
Then $N^{1/2}(\hat{\vv}_N-\vv^*)$ converges in distribution to a Gaussian random variable $\N(0,\sigma^2)$ with
\[\textstyle
\sigma^2=\var\left[ \sum\limits_{i=1}^{n+1}\lambda^*_i F(X,\gamma^*,\xi^*_i)\right].
\]
\end{corollary}

The results we have given in this Section describe the limiting distribution of the optimal value $\hat{\vv}_N$ and complement the earlier result obtained by \citet{sha08} in the convex-concave
case.

\setcounter{equation}{0}
\section{Asymptotics of optimal solutions}
\label{sec-sol}

We now turn to the more delicate question of asymptotics for the optimal solutions $\hat{\gamma}_N$.

\subsection{Preliminaries}

Asymptotics of optimal solutions $\hat{\gamma}_N$ of problem \eqref{intro-minimax-sam} are governed by  second-order properties of the minimax problem, and involve considerably stronger regularity conditions.
We make the following assumption.

\begin{assumption} \label{ass-2nd-1}
\mbox{}\vspace*{-6pt}
\begin{mylist}{(ii)}
\item [{(i)}] $\Gamma^* = \{\gamma^*\}$  is the singleton and $\Xi^*(\gamma^*) = \{\xi_1^*, \ldots ,\xi^*_k\}$ is finite.
\item [{(ii)}] The points $\gamma^*$ and $\xi_1^*, \ldots ,\xi^*_k$ are interior points of the respective sets $\Gamma$ and $\Xi$.
\end{mylist}
\end{assumption}

Part (i) of this assumption is admittedly quite restrictive but, together with part (ii), it allows us to write the original minimax problem \eqref{intro-minimax-pop} as a finite minimax problem.
This in turn makes it possible to reformulate the problem as a standard nonlinear programming problem with a finite number of inequality constraints.

With this goal in mind,
let $\cN$ be a compact and convex neighborhood of $\gamma^*$
that is contained in the interior of the set $\Gamma$.
Moreover, for $i=1,\ldots,k$, let
$\Xi_i^*$ be a compact neighborhood of $\xi_i^*$ such that
$\Xi_1^*,\ldots,\Xi_k^*$ are disjoint subsets of $\Xi$.
Define
\begin{equation}\label{varphi-2}
\varphi_i(\gamma) := \sup_{\xi\in \Xi_i^*} f(\gamma,\xi) = \sup_{\xi\in \Xi_i^*} \bbe[ F(X,\gamma,\xi) ],\quad i=1,\ldots,k.
\end{equation}
It now follows that the function
\[
 \varphi(\gamma):=\sup_{\xi\in \Xi} f(\gamma,\xi)
\]
can, by choosing $\cN$ sufficiently small, be for $\gamma\in \cN$ represented as
\[
   \varphi(\gamma) =\max \left \{\varphi_1(\gamma), \ldots ,\varphi_k(\gamma)\right\}.
\]
This approach is motivated by the so-called reduction method used in the second-order analysis of semi-infinite programs (cf.,  \citealp{HK1993}).

Therefore, under Assumption \ref{ass-2nd-1}, we can write the original population minimax problem \eqref{intro-minimax-pop} as the finite minimax problem
\begin{equation}\label{repr-2}
   \min_{\gamma\in \cN} \varphi(\gamma)
   = \min_{\gamma\in \cN}
   	\max \left \{\varphi_1(\gamma), \ldots ,\varphi_k(\gamma)\right\}.
\end{equation}
Note that
\[
 \varphi_1(\gamma^*)= \ldots = \varphi_k(\gamma^*)=\varphi(\gamma^*)=\vv^* .
\]

The empirical counterparts of functions $\varphi_i$  and $\varphi$  are
\begin{eqnarray}
\label{phiemp-1}
\hat{\varphi}_{iN}(\gamma)&:=&\sup_{\xi\in \Xi_i^*} \hat{f}_N(\gamma,\xi)
=\sup_{\xi\in \Xi_i^*} N^{-1}\tsum_{j=1}^N F(X_j,\gamma,\xi),  \quad i=1,\ldots,k,\\
\notag
\hat{\varphi}_{N}(\gamma)&:=& \max \left \{\hat{\varphi}_{1N}(\gamma), \ldots ,\hat{\varphi}_{kN}(\gamma)\right\}.
\end{eqnarray}
The sample minimax problem \eqref{intro-minimax-sam} can be written as
\begin{equation}\label{minimax-sam}
\min_{\gamma\in \Gamma} \hat{\varphi}_{N} (\gamma)
= \min_{\gamma\in \Gamma}
	\max \left \{\hat{\varphi}_{1N}(\gamma), \ldots ,\hat{\varphi}_{kN}(\gamma)\right\}
\end{equation}
and
\[
\hat{\vv}_N
= \inf_{\gamma\in  \cN} \hat{\varphi}_{N}(\gamma)
= \hat{\varphi}_N(\hat{\gamma}_N).
\]

Our subsequent analysis relies on second-order properties of the above problems, and for that  appropriate differentiability properties are needed.
In our next assumption, we require that $F(X,\gamma,\xi)$ is twice continuously differentiable in $(\gamma,\xi)$, and denote by $\nabla^2 F(X,\gamma,\xi)$ the respective $(n+m)\times (n+m)$  Hessian matrix of second-order partial  derivatives with respect to $(\gamma,\xi)$.

\begin{assumption} \label{ass-sol-differentiability}
\mbox{}\vspace*{-6pt}
\begin{mylist}{(iii)}
\item [{(i)}] For almost every $X$ the function $F(X,\gamma,\xi)$ is twice continuously differentiable in $(\gamma,\xi)$.
\item [{(ii)}] There is an integrable function $C(X)$ such that  almost surely     $\|\nabla^2 F(X,\gamma,\xi)\|\le C(X)$
for all $(\gamma,\xi)\in\Gamma\times\Xi$.
\item [{(iii)}] The matrices $\bbe[\nabla^2_{\xi\xi} F(X,\gamma^*,\xi_i^*)]$, $i=1,\ldots,k$,  are nonsingular.
\end{mylist}
\end{assumption}

It follows from Assumption \ref{ass-sol-differentiability} that the function $f(\gamma,\xi)$ is twice continuously differentiable with
\[
 \nabla^2 f(\gamma,\xi)=\bbe[\nabla^2 F(X,\gamma,\xi)]
\]
and that the Hessian matrices $\nabla_{\xi\xi}^2 f(\gamma^*,\xi_i^*)$, $i=1,\ldots,k$,  are nonsingular.
Moreover, it also follows that $\hat{f}_N(\gamma,\xi)$ is twice continuously differentiable.
By the uniform LLN, we also obtain the following uniform convergence results
 \begin{eqnarray}
 \label{LLN0}
 && \sup_{(\gamma,\xi)\in \Gamma\times \Xi} |  \hat{f}_N(\gamma,\xi)-  f (\gamma,\xi) | \to 0\;{\rm w.p.1}.\\
 \label{LLN1}
 && \sup_{(\gamma,\xi)\in \Gamma\times \Xi} \| \nabla  \hat{f}_N(\gamma,\xi)-\nabla  f (\gamma,\xi) \|\to 0\;{\rm w.p.1}.\\
  \label{LLN2}
&&  \sup_{(\gamma,\xi)\in \Gamma\times \Xi} \|\nabla^2 \hat{f}_N(\gamma,\xi)-\nabla^2 f (\gamma,\xi)\|\to 0\;{\rm w.p.1}.
 \end{eqnarray}
Furthermore, by the CLT we obtain that
$N^{1/2} \bigl( \nabla \hat{f}_N(\gamma^*,\xi_i^*)-\nabla  f (\gamma^*,\xi_i^*) \bigr)$
converges in distribution and thus
\begin{equation}
\label{CLT-f}
\|\nabla \hat{f}_N(\gamma^*,\xi_i^*)-\nabla  f (\gamma^*,\xi_i^*)\| = O_p(N^{-1/2}).
\end{equation}

Analogous differentiability properties are also needed for the functions
$\varphi_i(\gamma)$ and $\hat{\varphi}_{iN}(\gamma)$, $\gamma\in \cN$, $i=1,\ldots,k$,
defined in  \eqref{varphi-2} and \eqref{phiemp-1}.
These are established in the following lemma.

\begin{lemma}
\label{lem-secorder1}
Suppose that Assumptions {\rm \ref{ass-2nd-1}--\ref{ass-sol-differentiability}} are satisfied. Then the following results hold for $i=1,\ldots,k$ (where the neighborhood $\cN$ is redefined if necessary).
\vspace*{-6pt}
\begin{mylist}{(iii)}
\item [{(i)}] The function $\varphi_i(\cdot)$ is twice continuously differentiable on $\cN$ with
\begin{eqnarray}
\hspace*{-30pt} \nabla \varphi_i(\gamma^*)
   &=& \nabla_\gamma  f(\gamma^*,\xi_i^*), \label{gradf} \\
\hspace*{-30pt} \nabla^2 \varphi_i(\gamma^*)
   &=& \nabla_{\gamma\gamma}^2 f(\gamma^*,\xi_i^*)-
   \bigl[\nabla_{\gamma\xi}^2 f(\gamma^*,\xi_i^*)\bigr ]
   \bigl [\nabla_{\xi\xi}^2 f(\gamma^*,\xi_i^*)\bigr ]^{-1}
   \bigl [\nabla_{\xi\gamma}^2 f(\gamma^*,\xi_i^*)\bigr ].
   \label{hess-a}
\end{eqnarray}
\item [{(ii)}] With probability one for $N$ large enough, the  function $\hat{\varphi}_{iN}(\cdot)$ is twice  continuously differentiable on $\cN$.
\item [{(iii)}] The following convergence results hold:
\begin{eqnarray}
\label{uniform-0}
&& \sup_{\gamma\in \cN} |  \hat{\varphi}_{iN}(\gamma)- \varphi_{i}(\gamma) |  \to 0\;{\rm w.p.1},\\
\label{uniform-1}
&& \sup_{\gamma\in \cN} \|\nabla  \hat{\varphi}_{iN}(\gamma)- \nabla  \varphi_{i}(\gamma)\|  \to 0\;{\rm w.p.1},\\
\label{uniform-2}
&& \sup_{\gamma\in \cN} \|\nabla^2 \hat{\varphi}_{iN}(\gamma)- \nabla^2 \varphi_{i}(\gamma)\|  \to 0\;{\rm w.p.1},\\
\label{conv-dist-alt}
&& \nabla \hat{\varphi}_{iN}(\gamma^*)
	=\nabla_\gamma \hat{f}_{N}(\gamma^*,\xi_i^*) + o_p(N^{-1/2}), \\
\label{bound-in-prob-1}
&& \|\nabla \hat{\varphi}_{iN}(\gamma^*)-\nabla \varphi_i(\gamma^*)\| = O_p(N^{-1/2}).
\end{eqnarray}
\end{mylist}
\end{lemma}

\begin{proof} (i)
As $\xi_i^*$ is an interior point of $\Xi_i^*$ and the (unique) maximizer of $f(\gamma^*,\xi)$ over $\xi\in \Xi_i^*$,  by the first-order optimality condition $\nabla_\xi f(\gamma^*,\xi_i^*)=0$.
Recall that $\nabla_{\xi\xi}^2 f(\gamma^*,\xi_i^*)$ is nonsingular.
It now follows from the Implicit Function Theorem
(e.g., \citealp[Theorem 2.4.1]{Fiacco1983})
that for all $\gamma$ in a neighborhood of $\gamma^*$ the equation $\nabla_\xi f(\gamma,\cdot)=0$ has a unique solution $\xi_i(\gamma)$ over $\xi\in \Xi^*_i$,
that $f(\gamma,\xi)$ has a unique maximizer over $\xi\in \Xi^*_i$ at this $\xi_i(\gamma)$,
$\xi_i(\gamma^*) = \xi_i^*$,
and that $\xi_i(\gamma)$ is continuously differentiable.
Danskin's theorem (e.g., \citealp[Theorem 4.13]{BS}) now implies that
$\varphi_i(\cdot)$ is differentiable in a neighborhood of  $\gamma^*$
with $\nabla \varphi_i(\gamma) = \nabla_\gamma  f(\gamma,\xi_i(\gamma))$.
This establishes \eqref{gradf}.
Taking derivatives of $\nabla_\xi f(\gamma,\xi_i(\gamma))=0$ and
$\nabla \varphi_i(\gamma) = \nabla_\gamma  f(\gamma,\xi_i(\gamma))$
with respect to $\gamma$ and using the chain rule yields
\begin{eqnarray*}
&& \nabla^2_{\gamma\xi} f(\gamma,\xi_i(\gamma))
+ \nabla \xi_i(\gamma)  \nabla^2_{\xi\xi} f(\gamma,\xi_i(\gamma)) =0 \quad\text{and}\\
&& \nabla^2 \varphi_i(\gamma) = \nabla^2_{\gamma\gamma}  f(\gamma,\xi_i(\gamma)) +
\nabla \xi_i(\gamma) \nabla^2_{\xi\gamma}  f(\gamma,\xi_i(\gamma)) .
\end{eqnarray*}
Therefore in a neighborhood of $\gamma^*$,
\[
\nabla^2 \varphi_i(\gamma)
   = \nabla_{\gamma\gamma}^2 f(\gamma,\xi_i(\gamma))-
   \bigl[\nabla_{\gamma\xi}^2 f(\gamma,\xi_i(\gamma))\bigr ]
   \bigl [\nabla_{\xi\xi}^2 f(\gamma,\xi_i(\gamma))\bigr ]^{-1}
   \bigl [\nabla_{\xi\gamma}^2 f(\gamma,\xi_i(\gamma))\bigr ],
\]
establishing \eqref{hess-a} (cf. \citealp[Lemma 2.2]{sha1985}).

(ii) Since $\xi_i^*$ is assumed to be an interior point of $\Xi$,
we have that w.p.1 for $N$ large enough a point
$\hat{\xi}_{iN}^* \in \arg\max_{\xi \in \Xi^*_i}\hat{f}_N(\gamma^*,\xi)$
is an interior point of $\Xi$ and hence $\nabla_\xi \hat{f}_N(\gamma^*, \hat{\xi}_{iN}^* )=0$.
Moreover, as $\xi_i^*$ is the unique maximizer of $f(\gamma^*,\xi)$ over $\xi\in \Xi_i^*$,
uniform convergence result \eqref{LLN0} implies that $\hat\xi^*_{iN} \to \xi_i^*$ w.p.1 as $N\to\infty$.
By Assumption \ref{ass-sol-differentiability}(iii) and the LLN we also have that w.p.1 for $N$ large enough the Hessian matrix $\nabla_{\xi\xi}^2 \hat{f}_N(\gamma^*,\xi_i^*)$ is nonsingular,
and therefore also $\nabla_{\xi\xi}^2 \hat{f}_N(\gamma^*,\hat{\xi}_{iN}^* )$ is nonsingular w.p.1 for $N$ large enough.
We can now apply the Implicit Function Theorem as in part (i), and conclude that
w.p.1 for $N$ large enough and
for all $\gamma$ in a neighborhood of $\gamma^*$ the equation $\nabla_\xi \hat{f}_N(\gamma,\cdot)=0$ has a unique solution $\hat{\xi}_{iN}(\gamma)$ over $\xi\in \Xi^*_i$,
that $\hat{f}_N(\gamma,\xi)$ has a unique maximizer over $\xi\in \Xi^*_i$ at this $\hat{\xi}_{iN}(\gamma)$,
$\hat{\xi}_{iN}(\gamma^*) = \hat{\xi}_{iN}^*$,
and that $\hat{\xi}_{iN}(\gamma)$ is continuously differentiable.
Proceeding as in the proof of part (i), it follows that w.p.1 for $N$ large enough
$\hat{\varphi}_{iN}(\gamma)$ is twice continuously differentiable in a neighborhood of $\gamma^*$,
$ \nabla \hat{\varphi}_{iN}(\gamma)=\nabla_\gamma \hat{f}_{N}(\gamma,\hat\xi_{iN}(\gamma))$,
and
\begin{align*}
\nabla^2 \hat{\varphi}_{iN}(\gamma)
   &= \nabla_{\gamma\gamma}^2 \hat{f}_{N}(\gamma,\hat\xi_{iN}(\gamma)) \\
   & \quad-
   \bigl[\nabla_{\gamma\xi}^2 \hat{f}_{N}(\gamma,\hat\xi_{iN}(\gamma))\bigr ]
   \bigl [\nabla_{\xi\xi}^2 \hat{f}_{N}(\gamma,\hat\xi_{iN}(\gamma))\bigr ]^{-1}
   \bigl [\nabla_{\xi\gamma}^2 \hat{f}_{N}(\gamma,\hat\xi_{iN}(\gamma))\bigr ].
\end{align*}
This completes the proof of (ii).

(iii) Results \eqref{uniform-0}--\eqref{uniform-2}  follow from
\eqref{LLN0}--\eqref{LLN2} and the (nonsingularity) Assumption \ref{ass-sol-differentiability}(iii).
For \eqref{conv-dist-alt}, note that by
the proof of Lemma \ref{lem-secorder1}(ii),
$ \nabla \hat{\varphi}_{iN}(\gamma^*)=\nabla_\gamma \hat{f}_{N}(\gamma^*,\hat\xi_{iN}^*)$,
so that it suffices to show that
\[
\|\nabla_\gamma \hat{f}_{N}(\gamma^*,\hat\xi^*_{iN})-\nabla_\gamma \hat{f}_{N}(\gamma^*, \xi^*_{i})\|= o_p(N^{-1/2}).
\]
To this end, recall that by the proof of Lemma \ref{lem-secorder1}(ii),
$\hat\xi^*_{iN} \to \xi_i^*$ w.p.1 as $N\to\infty$.
We next show that
$\|\hat\xi^*_{iN}-\xi_i^*\|=O_p(N^{-1/2})$.
For this, note that Assumption \ref{ass-sol-differentiability}(iii) implies that the Hessian matrix
$\nabla_{\xi\xi}^2 f(\gamma^*,\xi_i^*)$ is negative definite and thus by second-order optimality conditions there exists a constant $c>0$ and a neighborhood $N_i$ of $\xi_i^*$ such that
\[
f(\gamma^*,\xi) \le f(\gamma^*,\xi_i^*) - c\| \xi-\xi_i^*\|^2
\]
for all $\xi \in N_i$. We may assume that $N_i \subset \Xi_i^*$.

Now set $B_{iN} := \{ \xi \in N_i : \| \xi-\xi_i^* \| \le  \| \hat\xi^*_{iN} - \xi_i^* \| \}$
and consider the Lipschitz constant $\kappa_N$ of the function
$\hat{f}_{N}(\gamma^*,\cdot) -  f(\gamma^*, \cdot)$
on $B_{iN}$.
For any $\xi \in B_{iN}$,
\begin{align*}
&\| \nabla_{\xi} \hat{f}_{N}(\gamma^*, \xi) - \nabla_{\xi} f(\gamma^*, \xi) \|    \\
&\qquad\le \| \nabla_{\xi} \hat{f}_{N}(\gamma^*, \xi_i^*) - \nabla_{\xi} f(\gamma^*, \xi_i^*) \|    \\
    &\qquad\quad + \| [\nabla_{\xi} \hat{f}_{N}(\gamma^*, \xi) - \nabla_{\xi} f(\gamma^*, \xi)] 
    - [\nabla_{\xi} \hat{f}_{N}(\gamma^*, \xi_i^*) - \nabla_{\xi} f(\gamma^*, \xi_i^*)] \|  . 
\end{align*}
By \eqref{CLT-f}, the first term is $O_p(N^{-1/2})$.
On the other hand, by the mean value theorem and the definition of $B_{iN}$,
w.p.1 for $N$ large enough
\[
\sup_{\xi \in B_{iN}}
\| [\nabla_{\xi} \hat{f}_{N}(\gamma^*, \xi) - \nabla_{\xi} f(\gamma^*, \xi)]
    - [\nabla_{\xi} \hat{f}_{N}(\gamma^*, \xi_i^*) - \nabla_{\xi} f(\gamma^*, \xi_i^*)] \|
   \le C_N
\| \hat\xi^*_{iN} - \xi_i^* \|,
\]
where
$C_N:= \sup_{\xi \in \Xi} \| \nabla^2_{\xi\xi} \hat{f}_{N}(\gamma^*, \xi) - \nabla^2_{\xi\xi} f(\gamma^*, \xi) \|$
and by \eqref{LLN2} $C_N = o_p(1)$.
Therefore
\[
\sup_{\xi \in B_{iN}} \| \nabla_{\xi} \hat{f}_{N}(\gamma^*, \xi) - \nabla_{\xi} f(\gamma^*, \xi) \|
\le o_p(1) \| \hat\xi^*_{iN} - \xi_i^* \| + O_p(N^{-1/2}),
\]
which provides an upper bound for the Lipschitz constants $\kappa_N$.

We now apply \citet[Proposition 4.32]{BS}, from which it follows that
\[  \|\hat\xi^*_{iN}-\xi_i^*\| \le o_p(1) \| \hat\xi^*_{iN} - \xi_i^* \| + O_p(N^{-1/2}) \]
and hence that
$\|\hat\xi^*_{iN}-\xi_i^*\| = O_p(N^{-1/2}) $.
Using this property, the mean value theorem to the function
$\nabla_\gamma \hat{f}_{N}(\gamma^*,\cdot)-\nabla_\gamma \hat{f}_{N}(\gamma^*, \cdot)$,
and \eqref{LLN2}, now imply that
\[
\|\nabla_\gamma \hat{f}_{N}(\gamma^*,\hat\xi^*_{iN})-\nabla_\gamma \hat{f}_{N}(\gamma^*, \xi^*_{i})\|= o_p(N^{-1/2}).
\]
Thus \eqref{conv-dist-alt} holds.

Finally, \eqref{bound-in-prob-1} follows from
\eqref{CLT-f}, \eqref{gradf}, and \eqref{conv-dist-alt}.
\end{proof}

\subsection{Optimality conditions for the population minimax problem}

We can now express the finite minimax problems of the previous subsection as nonlinear programming problems with a finite number of inequality constraints.
Specifically,
the population minimax problem \eqref{repr-2} can be equivalently formulated as the  problem (compare with \eqref{semi-f}):
\begin{equation}\label{equiv}
\min_{z\in \bbr, \gamma\in \Gamma} z \quad\text{s.t.}\quad \varphi_i(\gamma)-z\le 0,\;\; i=1,\ldots,k,
\end{equation}
and similarly for its empirical counterpart problem \eqref{minimax-sam}:
\begin{equation}\label{equiv-2}
\min_{z\in \bbr, \gamma\in \Gamma} z \quad\text{s.t.}\quad \hat{\varphi}_{iN}(\gamma)-z\le 0,\;\; i=1,\ldots,k.
\end{equation}
We next consider first- and second-order optimality conditions of problem \eqref{equiv}.

Consider the Lagrangian $L(z,\gamma,\lambda)=z+\sum_{i=1}^k\lambda_i (\varphi_i(\gamma)-z)$ of problem \eqref{equiv}.
As the appropriate Mangasarian-Fromovitz constraint qualification (see \citealp[p.~441]{BS}) is automatically satisfied, first-order optimality conditions imply
that there exist Lagrange multipliers $\lambda \in \bbr^k$ such that
\begin{equation}\label{lagrmul-2}
\lambda_i\ge 0,\; i=1,\ldots,k, \quad
\sum_{i=1}^{k}\lambda_i =1,
\quad\text{and}\quad
\sum_{i=1}^{k}\lambda_i \nabla  \varphi_i(\gamma^*)=0.
\end{equation}
These conditions can be viewed as a particular case of the optimality conditions \eqref{optcon}.
Denote by $\Lambda^*$ the set Lagrange multipliers $\lambda\in \bbr^k$ satisfying \eqref{lagrmul-2}.

Note that the directional derivative of the function $\varphi(\gamma) =\max \left \{\varphi_1(\gamma), \ldots ,\varphi_k(\gamma)\right\}$ is given by
\begin{equation}\label{phider}
\varphi'(\gamma^*,h)=\max_{1\le i\le k} h^\top \nabla  \varphi_i(\gamma^*),\;h\in \bbr^n.
\end{equation}
The first-order optimality condition \eqref{lagrmul-2} means that
$\varphi'(\gamma^*,h)\ge 0$ for all $h\in \bbr^n$.
Under condition \eqref{lagrmul-2},
the so-called  critical cone of problem \eqref{equiv} can be expressed as
\begin{equation}\label{critcone}
\C=\{h\in\mathbb{R}^n :  \varphi'(\gamma^*,h)\le 0\}.
\end{equation}
This cone represents directions where the first-order approximation does not give information about optimality of $\gamma^*$. Note that because of
$\varphi'(\gamma^*,h)\ge 0$, the inequality condition  in  \eqref{critcone} can be replaced by $\varphi'(\gamma^*,h)= 0$.

By \eqref{phider} the critical cone can equivalently be written   as
\[
\C=\{h\in\mathbb{R}^n : h^\top \nabla \varphi_i(\gamma^*)\le 0,\;i=1,\ldots,k\}.
\]
Alternatively, one can for any $\lambda\in \Lambda^*$ define the index sets
\[
\I_+ (\lambda):=\{i :\lambda_i>0,\; i=1,\ldots,k\}
\quad\text{and}\quad
\I_0 (\lambda):=\{i :\lambda_i=0,\; i=1,\ldots,k\}
\]
and express $\C$ as
\begin{equation}\label{critcone-2}
\C=\left \{h: h^\top \nabla \varphi_i(\gamma^*)= 0,\;i\in\I_+(\lambda),
\;\text{and}\;\;
h^\top \nabla \varphi_i(\gamma^*)\le 0 ,\;i\in\I_0(\lambda) \right\}
\end{equation}
(cf., \citealp[pp. 430, 443]{BS}).  Note that the right hand side of \eqref{critcone-2}  does not depend on the choice of $\lambda\in \Lambda^*$.

Turning to second-order conditions, note that
\[
 \nabla_{\gamma\gamma}^2 L(z,\gamma^*,\lambda)=\sum_{i=1}^k \lambda_i \nabla^2 \varphi_i(\gamma^*)
\]
does not depend on $z$.
Consider the following second-order condition:
\begin{equation}\label{secorder-1}
\sup_{\lambda\in \Lambda^*}
h^\top \nabla^2_{\gamma\gamma}L(z,\gamma^*,\lambda)h > 0, \;\forall h \in \C\setminus\{0\}.
\end{equation}
Under Assumptions \ref{ass-2nd-1}--\ref{ass-sol-differentiability},
condition \eqref{secorder-1} is necessary and sufficient for
the following quadratic growth condition to  hold.

\begin{definition}
It is said that
the {\em quadratic growth} condition holds for problem \eqref{repr-2} at $\gamma^*$ if there exist a constant   $c>0$ and a neighborhood $\V$ of $\gamma^*$  such  that
\begin{equation}\label{quadr}
 \varphi(\gamma)\ge \varphi(\gamma^*)+c\|\gamma-\gamma^*\|^2, \quad \forall\gamma\in \V\cap \Gamma.
\end{equation}
\end{definition}

We will also need a uniform version of the quadratic growth condition (cf., \citealp[Definition 5.16]{BS}).
Related to the population minimax problem \eqref{repr-2},
consider the more general problem
\[
  \min_{\gamma\in \Gamma} \phi(\gamma),
\]
where
\[ \phi(\gamma):=\max_{1\le i\le k} \phi_i(\gamma) \]
and $\phi_i:\bbr^n\to \bbr$, $i=1,\ldots,k$,  are twice   continuously  differentiable functions. Consider the $C^2$-norm  of twice continuously differentiable function $g:\bbr^n\to \bbr$;
\[
\|g\|_{C^2}:=\sup_{\gamma\in \cN}|g(\gamma) |+
\sup_{\gamma\in \cN}\|\nabla g(\gamma) \|+
\sup_{\gamma\in \cN}\|\nabla^2g(\gamma) \|,
\]
where $\cN$ is a compact neighborhood of $\gamma^*$.
The following definition formalizes a uniform variant of the quadratic growth condition for functions $\phi$ sufficiently close to $\varphi$.

\begin{definition}
 It is said that the {\em uniform quadratic growth} condition holds for problem \eqref{repr-2} at $\gamma^*$ if there exist constants $c>0$ and $\e>0$ and a neighborhood $\cN$ of $\gamma^*$ such that for $\| \phi_i-\varphi_i \|_{C^2}\le \e$, $i=1,\ldots,k$, and
 $\tilde{\gamma}\in \arg\min_{\gamma\in \cN\cap \Gamma}\phi(\gamma)$ it follows that
 \begin{equation}\label{unifromgrowth}
  \phi(\gamma)\ge \phi(\tilde{\gamma})+c \|\gamma-\tilde{\gamma}\|^2,\quad \forall \gamma \in \cN\cap \Gamma.
 \end{equation}
 \end{definition}

Of course the uniform quadratic growth condition implies the quadratic growth \eqref{quadr}.
In order to ensure the uniform quadratic growth condition we need a stronger form of second-order sufficient conditions.

 \begin{assumption}
\label{ass-interior-2}
The vectors $\nabla \varphi_1(\gamma^*)-\nabla \varphi_k(\gamma^*), \ldots ,\nabla \varphi_{k-1}(\gamma^*)-\nabla \varphi_k(\gamma^*)$ are linearly independent.
\end{assumption}

Assumption \ref{ass-interior-2}   means the   affine independence of vectors
$\nabla\varphi_1(\gamma^*), \ldots ,\nabla \varphi_k(\gamma^*)$,
that is,  if     $\sum_{i=1}^k \alpha_i \nabla \varphi_1(\gamma^*)=0$ and
$\sum_{i=1}^k \alpha_i=0$, then $\alpha_i=0$, $i=1,\ldots,k$.
It  implies that $k\le n+1$.
Under this assumption there exists a unique vector $\lambda^*$ of Lagrange multipliers, i.e., $\Lambda^*=\{\lambda^*\}$.

\begin{assumption}
\label{ass-interior-3}
The following condition holds
\begin{equation}\label{strsec}
 \begin{array}{l}
   h^\top\bigl [ \sum_{i=1}^k\lambda^*_i \nabla^2\varphi_i(\gamma^*)\bigr]  h>0 ,
   \quad \forall   h\in \LL\setminus\{0\},
    \end{array}
\end{equation}
where
\[
  \LL:=\left\{h\in \bbr^n:h^\top \nabla \varphi_i(\gamma^*) =0,\;i\in \I_+(\lambda^*)\right \}.
\]
\end{assumption}

Note that $\LL$ is a linear space containing the critical cone $\C$. Recall that by the first-order condition \eqref{lagrmul-2}, $\sum_{i\in\I_+(\lambda^*)}\lambda_i^* \nabla \varphi_i(\gamma^*)=0$.
Condition \eqref{strsec} can be viewed as a stringent form of second-order sufficient conditions. Of course this condition holds if the Hessian matrix $\sum_{i=1}^k\lambda^*_i \nabla^2\varphi_i(\gamma^*)$ is positive definite. Unlike the second-order condition \eqref{secorder-1}, condition \eqref{strsec} is stable with respect to small (with respect to the $C^2$-norm) perturbations of functions $\varphi_i$
(i.e., this condition is preserved under small, with respect to the $C^2$-norm, perturbations of the data;
cf., the discussion in  \citealp[p.374-5, eq. (4.368)]{BS}).

The preceding assumptions now ensure that the uniform quadratic growth condition holds.

 \begin{lemma}
 Suppose that Assumptions {\rm \ref{ass-2nd-1}--\ref{ass-interior-3}} hold.
 Then  the uniform quadratic growth condition  follows.
 \end{lemma}

\begin{proof}
We argue by a contradiction. Suppose that the uniform quadratic growth condition does not hold. Then for $i=1,\ldots,k$, there exists a sequence $\phi_{im}$,  $m=1,2,\ldots$, of functions converging in $C^2$-norm to $\varphi_i$ such that
  for $\phi_m(\cdot)=\max_{1\le i\le k} \phi_{im}(\cdot)$,
 $\tilde{\gamma}_m\in \arg\min_{\gamma\in \cN}\phi_m(\gamma)$ and some $ \gamma_m\in \cN$,  $ \gamma_m\ne \tilde{\gamma}_{m}$,  it follows that
 \begin{equation}\label{secineq}
  \phi_{m}(\gamma_m)\le \phi_{m}(\tilde{\gamma}_{m})+o(\|\gamma_m-\tilde{\gamma}_{m}\|^2),
 \end{equation}
 and $\tilde{\gamma}_m$
and $\gamma_m$ tend
to $\gamma^*$.
 Consider $h_m:=t_m^{-1}(\gamma_m-\tilde{\gamma}_m)$,
where $t_m:=\|\gamma_m-\tilde{\gamma}_m\|$
tends to $0$.
 By passing to a subsequence if necessary, we can assume that  $h_m$  converges to a point  $h$  with $\|h \|=1$.
For $m=1,2,\ldots$, consider the problem $\min_{\gamma\in\Gamma} \phi_m(\gamma)$ and let
 $\tilde{\lambda}_{im}$ denote the respective Lagrange multipliers.
 By the first-order optimality conditions we have that
 $\sum_{i=1}^k  \tilde{\lambda}_{im}=1$,
 $\sum_{i=1}^k  \tilde{\lambda}_{im}  \phi_{i m}(\tilde{\gamma}_{m})= \phi_{m}(\tilde{\gamma}_{m})$ and
$\sum_{i=1}^k  \tilde{\lambda}_{im} \nabla  \phi_{i m}(\tilde{\gamma}_{m})=0$.

Making use of a second-order Taylor expansion, we can write
\begin{equation}
\label{sequence}
\begin{split}
 \phi_{m}(\gamma_m)&= \phi_{m}(\tilde{\gamma}_{m}+t_m h_m) =
 \max\limits_{1\le i\le k} \phi_{im}(\tilde{\gamma}_{m}+t_m h_m)
 \\
 &= \max\limits_{1\le i\le k} 
 \bigl\{ \phi_{i m}(\tilde{\gamma}_{m})+
 t_m h_m^\top \nabla  \phi_{i m}(\tilde{\gamma}_{m})+\half t_m^2
  h_m^\top \nabla^2  \phi_{i m}(\tilde{\gamma}_{m})h_m
 \bigr\}
 + o(t^2_m)\\
 &\ge {\textstyle \sum\limits_{i=1}^k}  \tilde{\lambda}_{im} 
 \bigl\{ \phi_{i m}(\tilde{\gamma}_{m})+
 t_m h_m^\top \nabla  \phi_{i m}(\tilde{\gamma}_{m})+\half t_m^2
  h_m^\top \nabla^2  \phi_{i m}(\tilde{\gamma}_{m})h_m
 \bigr\}
 + o(t^2_m).
 \end{split}
\end{equation}
Consequently    by \eqref{secineq}, \eqref{sequence}, and the first-order optimality conditions we obtain
\[ \textstyle 0\ge
    h_m^\top\biggl[ \sum\limits_{i=1}^k  \tilde{\lambda}_{im}
   \nabla^2  \phi_{i m}(\tilde{\gamma}_{m})\biggr]h_m
 + o(1),
\]
and hence  by    passing to the limit
\begin{equation}\label{secineq-3}\textstyle
 h^\top\biggl [ \sum\limits_{i=1}^k\lambda^*_i \nabla^2\varphi_i(\gamma^*)\biggr]h\le 0.
\end{equation}

On the other hand, it also follows from   \eqref{secineq} and \eqref{sequence} that
$\phi_{m}(\tilde{\gamma}_{m}+t_m h_m) - \phi_{m}(\tilde{\gamma}_{m}) \le o(\| t_m h_m \|^2)$
so that
\[
\frac{\phi_{m}(\tilde{\gamma}_{m}+t_m h_m) - \phi_{m}(\tilde{\gamma}_{m})
- t_m h_m^\top \nabla  \phi_{m}(\tilde{\gamma}_{m})}{\| t_m h_m \|}
+ h_m^\top \nabla  \phi_{m}(\tilde{\gamma}_{m})
\le o(\| t_m h_m \|)
\]
and therefore
$$
\max_{1\le i\le k} h_m^\top \nabla  \phi_{i m}(\tilde{\gamma}_{m})+ o(1)\le 0,
$$
and hence by going to the limit that $\max_{1\le i\le k} h^\top \nabla  \varphi_i (\gamma^*)\le 0$. That is,
 \begin{equation}\label{secineq-2}
 \varphi'(\gamma^*,h)\le 0.
 \end{equation}
 Moreover, by continuity the Lagrange multipliers remain nonzero for small perturbations
and hence      for such  $h_m$ we have that
 $h_m^\top \nabla  \phi_{i m}(\tilde{\gamma}_{m})=0$, $i\in \I_+(\lambda^*)$.
 It follows that $h^\top \nabla  \varphi_{i}(\gamma^*)=0$, $i\in \I_+(\lambda^*)$. Together with \eqref{secineq-2} this implies that   $h\in \LL$.
 By \eqref{secineq-3} this gives the
  contradiction with the  second-order condition \eqref{strsec}.
\end{proof}

\subsection{Quadratic approximation of the sample minimax problem}

We next show that, when the uniform quadratic growth condition holds,
the solutions of the sample minimax problem \eqref{minimax-sam}
are close to the solutions of a particular quadratic approximation of the sample minimax problem. To this end, consider the problem
 \begin{equation}\label{approxprob}
  \min_{\gamma\in\V\cap \Gamma}  \psi_N(\gamma),
 \end{equation}
 where
 \[ \psi_N(\gamma):=\max_{1\le i\le k} \psi_{iN}(\gamma) \]
 with
 \[
 \psi_{iN}(\gamma):= (\gamma-\gamma^*)^\top\nabla \hat{\varphi}_{iN}(\gamma^*)+\half
  (\gamma-\gamma^*)^\top\nabla^2 \varphi_i(\gamma^*) (\gamma-\gamma^*).
  \]
Note that $\psi_{iN}(\gamma)$ comprises the first- and second-order terms of a Taylor approximation of $\hat{\varphi}_{iN}(\gamma) - \hat{\varphi}_{iN}(\gamma^*)$ at $\gamma^*$.
Let $\bar{\gamma}_N$ denote the optimal solution of problem \eqref{approxprob}.
In what follows, we establish a connection between $\hat{\gamma}_N$ and $\bar{\gamma}_N$.

As a preliminary step, we establish the rate of convergence of the solutions
$\hat{\gamma}_N$ and $\bar{\gamma}_N$.

\begin{lemma}
\label{lem-conv}
Suppose that Assumptions  {\rm \ref{ass-2nd-1}--\ref{ass-sol-differentiability}} and the quadratic growth  condition \eqref{quadr}  hold. Then
\begin{equation}\label{gamma_hat_rate}
\|\hat{\gamma}_N-\gamma^*\|=O_p(N^{-1/2})
\end{equation}
and
\[
\|\bar{\gamma}_N-\gamma^*\|=O_p(N^{-1/2}).
\]
\end{lemma}

By Lemma \ref{lem-conv}  we have that   under Assumptions \ref{ass-2nd-1}--\ref{ass-sol-differentiability}, the second-order condition \eqref{secorder-1} implies that $\hat{\gamma}_N$ and $\bar{\gamma}_N$ both converge to $\gamma^*$ at the rate of $O_p(N^{-1/2})$.

\begin{proof}
First let us observe that  $\hat{\gamma}_N\to \gamma^*$ w.p.1 as $N\to \infty$ (cf., \citealp[Theorems 5.3 and 5.14]{SDR}).
For $r>0$ consider the set
\begin{equation}\label{neighb}
B_r:=\{\gamma\in \Gamma:\|\gamma-\gamma^*\|\le r\},
\end{equation}
  the sequence
\[
 r_N:=\|\hat{\gamma}_N-\gamma^*\|,
\]
and the corresponding sequence of (asymptotically shrinking) neighborhoods
$\V_N:=\V \cap B_{r_N}$.
Let $\kappa_N$ denote the Lipschitz constant of the function
$\hat{\varphi}_N(\cdot)-\varphi(\cdot)$ on $\V_N$.
Provided that $\hat{\gamma}_N\in\V$, it now follows from
the quadratic growth condition, Lipschitz continuity, and
\citet[Proposition 4.32]{BS} that
\begin{equation}\label{quad-2}
\|\hat{\gamma}_N-\gamma^*\|\le c^{-1}\kappa_N.
\end{equation}

We next provide an upper bound for $\kappa_N$.
Note that for any $i \in \{ 1,\ldots,k\}$ and any $\gamma\in B_{r_N}$
\begin{align*}
\|\nabla \hat{\varphi}_{iN}(\gamma)-\nabla  \varphi_{i}(\gamma)\|
& \le
\|\nabla  \hat{\varphi}_{iN}(\gamma^*)-\nabla  \varphi_{i}(\gamma^*)\|  \\
&\quad +
\| [\nabla \hat{\varphi}_{iN}(\gamma)-\nabla  \varphi_{i}(\gamma)]
   - [\nabla \hat{\varphi}_{iN}(\gamma^*)-\nabla  \varphi_{i}(\gamma^*)] \|.
\end{align*}
By \eqref{bound-in-prob-1}, the first term on the majorant side is $O_p(N^{-1/2})$.
On the other hand, by the mean value theorem and the definition of $B_{r_N}$,
w.p.1 for $N$ large enough
\[
\sup_{\gamma\in B_{r_N}}
\| [\nabla \hat{\varphi}_{iN}(\gamma)-\nabla  \varphi_{i}(\gamma)]
   - [\nabla \hat{\varphi}_{iN}(\gamma^*)-\nabla  \varphi_{i}(\gamma^*)] \|
   \le C_N
\|\hat{\gamma}_N-\gamma^*\|,
\]
where
$C_N:= \sup_{\gamma\in \cN} \|\nabla^2 \hat{\varphi}_{iN}(\gamma) - \nabla^2 \varphi_i(\gamma) \|$
and by \eqref{uniform-2} $C_N = o_p(1)$.
Therefore
\[
\max_{1\le i\le k}  \sup_{\gamma\in B_{r_N}}\|\nabla \hat{\varphi}_{iN}(\gamma)-\nabla  \varphi_{i}(\gamma)\|\le
o_p(1)  \|\hat{\gamma}_N-\gamma^*\|+O_p(N^{-1/2}),
\]
which provides an upper bound for the Lipschitz constants $\kappa_N$.
Together with  \eqref{quad-2} this implies \eqref{gamma_hat_rate}.

Now consider $\bar{\gamma}_N$ using similar arguments.
Specifically, for the Lipschitz continuity of $\varphi(\cdot) - \psi_N(\cdot)$ note that by the definition of $\psi_{iN}(\gamma)$
\[
 \nabla \varphi_{i}(\gamma)-\nabla  \psi_{iN}(\gamma) =
  \nabla \varphi_{i}(\gamma)-\nabla  \hat{\varphi}_{iN}(\gamma^*)
  -\nabla^2 \varphi_{i}(\gamma^*)(\gamma-\gamma^*)
\]
for $i = 1,\ldots,k$.
The right-hand side of this equation can be written as
\[
  \nabla \varphi_{i}(\gamma^*) - \nabla  \hat{\varphi}_{iN}(\gamma^*)
  + \nabla \varphi_{i}(\gamma)-\nabla \varphi_{i}(\gamma^*)
  -\nabla^2 \varphi_{i}(\gamma^*)(\gamma-\gamma^*)
\]
where using the mean value theorem
$\nabla \varphi_{i}(\gamma)-\nabla \varphi_{i}(\gamma^*)
= \nabla^2 \varphi_{i}(\gamma^{\bullet})(\gamma-\gamma^*)$ for some intermediate point $\gamma^{\bullet}$
between $\gamma$ and $\gamma^*$.
By \eqref{bound-in-prob-1} and the continuity of $\nabla^2 \varphi_{i}(\cdot)$ it follows that
$\sup_{\gamma\in B_{r_N}} \| \nabla \varphi_{i}(\gamma)-\nabla  \psi_{iN}(\gamma) \|$ is
$O_p(N^{-1/2}) + o_p(1)  \|\hat{\gamma}_N-\gamma^*\|$.
Similarly as above,
it now follows that $\|\bar{\gamma}_N-\gamma^*\|=O_p(N^{-1/2})$.
\end{proof}

The next proposition shows that $\hat{\gamma}_N$ and $\bar{\gamma}_N$
are asymptotically equivalent.

 \begin{proposition}
Suppose that Assumptions  {\rm \ref{ass-2nd-1}--\ref{ass-sol-differentiability}} and the uniform quadratic growth  condition   hold. Then
\[
\|\hat{\gamma}_N- \bar{\gamma}_N\|=o_p(N^{-1/2}).
\]
\end{proposition}

 \begin{proof}
First consider the uniform quadratic growth condition with the functions $\phi=\hat{\varphi}_{N}$.
By \eqref{uniform-0}--\eqref{uniform-2} we have
\[
 \|\hat{\varphi}_{iN}-\varphi_i\|_{C^2}=o_p(1),\quad i=1,\ldots,k.
\]
Therefore by the uniform quadratic growth condition we have  that the quadratic  growth condition \eqref{unifromgrowth} holds w.p.1
at $\tilde{\gamma}=\hat{\gamma}_N$
with $\phi_i=\hat{\varphi}_{iN}$, $i=1,\ldots,k$, and $\phi=\hat{\varphi}_{N}$
for sufficiently large $N$.
That is,
\[
\hat{\varphi}_{N}(\gamma) \ge \hat{\varphi}_{N}(\hat{\gamma}_N)
+ c \| \gamma - \hat{\gamma}_N \|^2,\;\forall \gamma \in \cN\cap \Gamma.
\]

Next we provide an upper bound for the Lipschitz constant of $\hat{\varphi}_{N}-\psi_N$. By the definition of $\psi_{iN}(\gamma)$,
 \begin{equation}\label{gradappr}
 \nabla \hat{\varphi}_{iN}(\gamma)-\nabla  \psi_{iN}(\gamma)=
  \nabla \hat{\varphi}_{iN}(\gamma)-\nabla  \hat{\varphi}_{iN}(\gamma^*)
  -\nabla^2 \varphi_{i}(\gamma^*)(\gamma-\gamma^*).
 \end{equation}
As both $\|\hat{\gamma}_N-\gamma^*\|$ and $\|\bar{\gamma}_N-\gamma^*\|$ are $O_p(N^{-1/2})$,
we can restrict the analysis to the neighborhood  $B_{r_N}$, defined in \eqref{neighb}, with $r_N=O_p(N^{-1/2})$.
  By \eqref{gradappr} and the mean value theorem applied to $\nabla \hat{\varphi}_{iN}(\cdot)$ we have
\begin{align*}
& \sup\limits_{\gamma\in B_{r_N}}\| \nabla \hat{\varphi}_{iN}(\gamma)-\nabla  \psi_{iN}(\gamma)\| \\
&\qquad\le  r_N  \sup\limits_{\gamma\in B_{r_N}}\|\nabla^2 \hat{\varphi}_{iN}(\gamma)-\nabla^2  \varphi_{i}(\gamma^*)\| \\
&\qquad\le  r_N\sup\limits_{\gamma\in B_{r_N}}  \big(
 \|\nabla^2 \hat{\varphi}_{iN}(\gamma)-\nabla^2 \varphi_{i }(\gamma)\|
 +\|\nabla^2 \varphi_{i }(\gamma)-\nabla^2  \varphi_{i}(\gamma^*)\|\big).
\end{align*}
 By \eqref{uniform-2} and   since $\varphi_{i}$ are twice continuously differentiable  it follows that
 \[
  \sup\limits_{\gamma\in B_{r_N}}\| \nabla \hat{\varphi}_{iN}(\gamma)-\nabla  \psi_{iN}(\gamma)\|  \le  r_N\, o_p(1),
 \]
 and   since $r_N=O_p(N^{-1/2})$,
  \[
 \sup_{\gamma\in B_{r_N}}\| \nabla \hat{\varphi}_{iN}(\gamma)-\nabla  \psi_{iN}(\gamma)\|= o_p(N^{-1/2}).
 \]
 Therefore the Lipschitz constant of $\hat{\varphi}_{N}-\psi_N$, say $\tilde{\kappa}_N$, satisfies $\tilde{\kappa}_N = o_p(N^{-1/2})$.

  The proof now is completed by using \citet[Proposition 4.32]{BS},
  from which is follows that
$\|\hat{\gamma}_N- \bar{\gamma}_N\| \le c^{-1} \tilde{\kappa}_N = o_p(N^{-1/2})$.
\end{proof}

The above result reduces the asymptotic analysis of optimal solution $\hat{\gamma}_N$ of the original sample minimax problem \eqref{minimax-sam} to the asymptotic analysis of solution $\bar{\gamma}_N$ of problem \eqref{approxprob}.

\subsection{Asymptotic distribution of the minimax solutions}

We next consider asymptotics for the optimal solutions $\bar{\gamma}_N$ of problem \eqref{approxprob}.
To this end, we analyze a slightly more general problem.
 Let $\cN\subset \Gamma$  be a compact neighborhood of $\gamma^*$ and
  consider problem
  \begin{equation}\label{equivform}
 \begin{array}{cll}
 \min\limits_{\zeta\in \bbr, \gamma\in \cN }& \zeta\\
 {\rm s.t.}& (\gamma-\gamma^*)^\top\big(\nabla \varphi_{i}(\gamma^*)+v_{i}\big) +\half
  (\gamma-\gamma^*)^\top\nabla^2 \varphi_i(\gamma^*)(\gamma-\gamma^*) \le \zeta,\:\;\; i=1,\ldots,k,
 \end{array}
 \end{equation}
 parameterized by vectors $v_i\in \bbr^n$, $i=1,\ldots,k$.
 Denote by  $\bar{\gamma}(v)$ an optimal solution of problem \eqref{equivform} considered as a function of the $nk\times 1$ vector $v=[v_1,\ldots,v_k]$.
Note that for
  $v_{i}:=\nabla \hat{\varphi}_{iN}(\gamma^*)-\nabla \varphi_{i}(\gamma^*)$, problem \eqref{equivform} is equivalent to problem \eqref{approxprob}.
We next analyze the differentiability of $\bar{\gamma}(v)$ at $v=0$.

Let $\LL(\zeta,\gamma,\lambda,v)$ be the Lagrangian of problem \eqref{equivform} so that
\[
\LL(\zeta,\gamma,\lambda,v) = \zeta + \sum_{i=1}^k \lambda_i
[ (\gamma-\gamma^*)^\top\big(\nabla \varphi_{i}(\gamma^*)+v_{i}\big) +\half
  (\gamma-\gamma^*)^\top\nabla^2 \varphi_i(\gamma^*)(\gamma-\gamma^*) - \zeta ].
\]
Straightforward differentiation yields $\nabla_{\zeta} \LL(\zeta,\gamma,\lambda,v) = 1- \sum_{i=1}^k \lambda_i$,
\begin{align*}
\nabla_{\gamma} \LL(\zeta,\gamma,\lambda,v) &= \sum_{i=1}^k \lambda_i [ \big(\nabla \varphi_{i}(\gamma^*)+v_{i}\big)
   + \nabla^2 \varphi_i(\gamma^*)(\gamma-\gamma^*) ] ,  \\
\nabla_{\lambda_i} \LL(\zeta,\gamma,\lambda,v) &= (\gamma-\gamma^*)^\top\big(\nabla \varphi_{i}(\gamma^*)+v_{i}\big) +\half
  (\gamma-\gamma^*)^\top\nabla^2 \varphi_i(\gamma^*)(\gamma-\gamma^*) - \zeta.
\end{align*}
Note that for $v=0$ and $(\zeta,\gamma)=(0,\gamma^*)$, the Lagrange multipliers in \eqref{lagrmul-2} satisfy the first-order conditions.
Note that for $v=0$  the set of Lagrange multipliers of problem \eqref{equivform} is $\Lambda^*$, and
  that if the second-order sufficient condition \eqref{secorder-1} holds and the neighborhood $\cN$ is sufficiently small,
  then $\bar{\gamma}(0)=\gamma^*$.

Next note that
$\nabla^2_{\gamma\gamma}\LL(\zeta,\gamma,\lambda,v)
=\sum_{i=1}^k \lambda_i  \nabla^2 \varphi_i(\gamma^*)$
and
$\nabla^2_{\gamma v_i} \LL(\zeta,\gamma,\lambda,v) = \lambda_i$
so that
\[
   D^2_{\gamma v}\LL(\zeta,\gamma,\lambda,v)(d\gamma,d v)
   =  d\gamma^\top\big(\tsum_{i=1}^k\lambda_i dv_i\big).
\]
By \citet[Theorem  4.137]{BS} (see also \citealp[Theorem 5.53, eq. (5.110) and Remark 5.55]{BS})
it follows that under Assumptions \ref{ass-2nd-1}--\ref{ass-interior-3},
$\bar{\gamma}(v)$ is
Fr\'echet (and Hadamard)
directionally differentiable\footnote{Note that in the finite dimensional case, Hadamard directional differentiability implies Fr\'echet directional differentiability, and conversely  Fr\'echet directional differentiability together with continuity of the directional derivative with respect to the direction implies Hadamard directional differentiability.}
at $v=0$ with the directional derivative
$\bar{\bgamma}'(0,h)$, $h=[h_1,\ldots,h_k]$ ($nk\times 1$),
given by the optimal solution $\bar{\eta}=\bar{\eta}(h)$ of the following  problem
\begin{equation}\label{minsol-3alt}
\begin{array}{cll}
 \min\limits_{\zeta\in\bbr, \eta\in\bbr^n} & \eta^\top \big(\sum_{i=1}^k \lambda_i^* h_i\big)
 +\half   \eta^\top\big  [\sum_{i=1}^k \lambda^*_i  \nabla^2 \varphi_i(\gamma^*)\big ] \eta\\
 {\rm s.t.}& \eta^\top  \nabla \varphi_i(\gamma^*) -\zeta =0 ,\quad i\in \I_+(\lambda^*), \\
 & \eta^\top \nabla \varphi_i(\gamma^*) -\zeta    \le 0 ,\quad i\in \I_0(\lambda^*).
 \end{array}
 \end{equation}
Note that since $\sum_{i\in\I_+(\lambda^*) }^k \lambda_i^* \nabla \varphi_i(\gamma^*)=0$, the variable $\zeta$ in problem \eqref{minsol-3alt} must satisfy $\zeta=0$   and can thus be removed. Therefore equivalently  problem  \eqref{minsol-3alt} can be written as
\[
\begin{array}{cll}
 \min\limits_{\eta\in\bbr^n} & \eta^\top \big(\sum_{i=1}^k \lambda_i^* h_i\big)
 +\half   \eta^\top\big  [\sum_{i=1}^k \lambda^*_i  \nabla^2 \varphi_i(\gamma^*)\big ] \eta\\
 {\rm s.t.}& \eta^\top  \nabla \varphi_i(\gamma^*)   =0,\quad i\in \I_+(\lambda^*), \\
 & \eta^\top \nabla \varphi_i(\gamma^*)     \le 0,\quad i\in \I_0(\lambda^*).
 \end{array}
 \]

To convert the preceding results to asymptotics, 
now consider the $nk\times 1$ vector
 $Z_N:=[Z_{1N},\ldots,Z_{kN}]$ with
 \begin{equation}\label{zeta}
  Z_{iN}:=  \nabla \hat{\varphi}_{iN}(\gamma^*)-\nabla \varphi_{i}(\gamma^*), \quad i=1,\ldots,k.
 \end{equation}
Note that $\bar{\gamma}(Z_N) = \bar{\gamma}_N$.
Also by \eqref{conv-dist-alt}
\[ Z_{iN}
=\nabla_\gamma \hat{f}_{N}(\gamma^*,\xi_i^*) - \nabla_\gamma f(\gamma^*,\xi_i^*)
+o_p(N^{-1/2})
\]
and thus by the CLT we have that $N^{1/2}Z_N$ converges in distribution to
$\cZ$ that has the multivariate normal distribution $\N(0,\Sigma)$ with the covariance matrix $\Sigma$ given by  the covariance matrix of the $nk\times 1$ vector
$[ \nabla_\gamma F(X,\gamma^*,\xi_1^*) , \ldots , \nabla_\gamma F(X,\gamma^*,\xi_k^*) ] $.

By the derivations so far it holds that
\[ \hat{\gamma}_N - \gamma^*
    = \bar{\gamma}_N - \gamma^* + o_p(N^{-1/2})
    = \bar{\gamma}(Z_N) - \bar{\gamma}(0) + o_p(N^{-1/2}). \]
Finally, by applying the (finite dimensional) Delta Theorem to problem \eqref{approxprob} we obtain the main result of our paper, characterizing the asymptotic behavior of the optimal solutions $\hat{\gamma}_N$.

\begin{theorem}
 \label{th-gamma}
Suppose that Assumptions {\rm \ref{ass-2nd-1} -- \ref{ass-interior-3}} hold. Then
\begin{equation}
\label{optasym}
\hat{\gamma}_N-\gamma^*  =  \tilde{\eta} ( Z_{N} ) + o_p(N^{-1/2}),
\end{equation}
where $Z_N=[Z_{1N},\ldots,Z_{kN}]$ with $Z_{iN}$ defined in \eqref{zeta} and
$\tilde{\eta}(Z_N)$ is the optimal solution of problem
\begin{equation}\label{minsol-3a}
\begin{array}{cll}
 \min\limits_{\eta\in\bbr^n} &
 \eta^\top \big  (\sum_{i=1}^k \lambda^*_i Z_{iN} \big)
 +\half   \eta^\top  \bigl  [\sum_{i=1}^k \lambda^*_i  \nabla^2 \varphi_i(\gamma^*)\bigr  ] \eta\\
 {\rm s.t.}& \eta^\top \nabla \varphi_i(\gamma^*)  =0,\quad i\in \I_+(\lambda^*), \\
 &\eta^\top \nabla \varphi_i(\gamma^*)      \le 0,\quad i\in \I_0(\lambda^*).
\end{array}
\end{equation}
\end{theorem}

Problem \eqref{minsol-3a} is a quadratic programming problem.
Under Assumptions \ref{ass-2nd-1} -- \ref{ass-interior-3}
the solution mapping $\tilde{\eta}(\cdot)$  is well defined and is   positively homogeneous. Therefore it follows from \eqref{optasym} that
 \begin{equation}\label{optas-2}
  N^{1/2}(   \hat{\gamma}_N-\gamma^*)\,\dst \, \tilde{\eta}(\cZ).
   \end{equation}
In general, the mapping $\tilde{\eta}(\cdot)$ is not necessarily linear and therefore this limiting distribution is typically non-Gaussian and rather complicated.
Nevertheless, it describes the asymptotic behavior of the optimal solutions $\hat{\gamma}_N$.

However, under some additional conditions much more transparent results can be obtained. Consider the following condition.

\begin{definition}
It is said that the {\em strict complementarity condition} holds if
$\lambda_i^*>0$ for $i=1,\ldots,k$, i.e., $\I_+(\lambda^*)=\{1,\ldots,k\}$.
\end{definition}

Suppose that in addition to the assumptions of Theorem \ref{th-gamma}
also the strict complementarity condition holds so that $\I_0(\lambda^*)=\emptyset$.
In this case the feasibility constraints of problem \eqref{minsol-3a} involve only equalities and define the linear space $\LL$ of dimension $n-k+1$ (recall that $\sum_{i=1}^k \lambda_i^*  \nabla \varphi_i(\gamma^*)  =0$).
It turns out that in this case the
mapping $\tilde{\eta}(\cdot)$ is linear, and consequently
$N^{1/2}( \hat{\gamma}_N-\gamma^*)$ converges in distribution to multivariate normal with  (degenerate) distribution concentrated on the linear space $\LL$.
To provide some detail,
note that in this case the first-order optimality conditions for problem \eqref{minsol-3a} can be written in the form of equations
\begin{equation}\label{firstorder}
  Y_N+H\eta+A \alpha=0,\quad  A^\top \eta=0,
\end{equation}
where  $\alpha$ is the vector of Lagrange multipliers associated with problem \eqref{minsol-3a},
$Y_N:=\sum_{i=1}^k \lambda^*_i Z_{iN}$,
$H:=\sum_{i=1}^k \lambda^*_i  \nabla^2 \varphi_i(\gamma^*)$, and
$A=[\nabla \varphi_1(\gamma^*), \ldots ,\nabla \varphi_{k-1}(\gamma^*)]$
is an $n\times (k-1)$ matrix
(with $k=|\I_+(\lambda^*)|$;
recall that $\sum_{i\in\I_+(\lambda^*)}\lambda_i^* \nabla \varphi_i(\gamma^*)=0$, and  therefore the last constraint $\eta^\top  \nabla \varphi_k(\gamma^*)   =0$ of problem \eqref{minsol-3a}  can be removed).
Furthermore, equations \eqref{firstorder} can be written as
\begin{equation*}
 \left[
 \begin{array}{ccc}
   H & A   \\
  A^\top & 0
 \end{array}\right]
 \left[
 \begin{array}{ccc}
  \eta   \\
  \alpha
 \end{array}\right]=
 \left[
 \begin{array}{ccc}
   -Y_N  \\
   0
 \end{array}\right],
 \end{equation*}
where, by Assumptions \ref{ass-interior-2}  and \ref{ass-interior-3}, the matrix on the left-hand side is nonsingular.
 It follows that in the considered case,  $N^{1/2}(   \hat{\gamma}_N-\gamma^*)$ converges in distribution to normal with zero mean and covariance matrix given by the $n\times n$ upper block of matrix
 \[
  \left[
 \begin{array}{ccc}
   H & A   \\
  A^\top & 0
 \end{array}\right]^{-1}
 \left[
 \begin{array}{ccc}
  \Sigma &0 \\
  0& 0
 \end{array}\right]
   \left[
 \begin{array}{ccc}
   H & A   \\
  A^\top & 0
 \end{array}\right]^{-1}.
 \]

To further illustrate, consider the even more specific,
rather simple, case of $k=1$.
In this case the minimax problem has a single unique solution, that is, $\Gamma^* = \{ \gamma^*\}$ and
$\Xi^*(\gamma^*)=\{\xi^*\}$.
In this case the function $\varphi(\gamma)$ is given by $\varphi(\gamma) = \sup_{\xi\in \bar{\Xi}} f(\gamma,\xi)$, where $\bar{\Xi}$  is a compact neighborhood of $\xi^*$. The above matrix $A$ vanishes, $\lambda_1^*=1$,  and problem \eqref{minsol-3a} becomes an unconstrained problem. In this case  $N^{1/2}(   \hat{\gamma}_N-\gamma^*)$ converges in distribution to normal with zero mean and covariance matrix
$H^{-1}\Sigma H^{-1}$
with
\[
H= \nabla^2 \varphi(\gamma^*)= \nabla_{\gamma\gamma}^2 f(\gamma^*,\xi^*)-
   \bigl[\nabla_{\gamma\xi}^2 f(\gamma^*,\xi^*)\bigr ]
   \bigl [\nabla_{\xi\xi}^2 f(\gamma^*,\xi^*)\bigr ]^{-1}
   \bigl [\nabla_{\xi\gamma}^2 f(\gamma^*,\xi^*)\bigr ]
\]
and $\Sigma=\cov[\nabla_\gamma F(X,\gamma^*,\xi^*)]$.

Theorem \ref{th-gamma} gives the asymptotics of $\hat{\gamma}_N$ under a set of rather general assumptions.
We close this section with remarks about two situations in which some changes to this result are needed.

\begin{remark}
In Assumption \ref{ass-2nd-1} it was assumed that $\Xi^*(\gamma^*) = \{\xi_1^*, \ldots ,\xi^*_k\}$
with the points
$\xi_1^*, \ldots ,\xi^*_k$ being interior points of $\Xi$.
If, on the contrary, a point
$\xi^*_i$ is an {\em isolated} point of $\Xi$, then formula \eqref{hess-a} should be replaced by
$\nabla^2 \varphi_i(\gamma^*)=
    \nabla_{\gamma\gamma}^2 f(\gamma^*,\xi_i^*)$  throughout the derivations.
For instance, if the set $\Xi$ itself is finite,
then of course every point of $\Xi$ is isolated.
\end{remark}

\begin{remark}
Now consider the rather specific case in which
 {\em $k=n+1$ holds as an equality}
 (recall that Assumption \ref{ass-interior-2} implies that $k \leq n+1$).
 Moreover, suppose that
Assumptions  \ref{ass-2nd-1} -- \ref{ass-interior-2} are satisfied, and that the strict complementarity condition  holds.
In this case,
$\varphi'(\gamma^*,h)>0$ for all $h\ne 0$, the linear space
$\LL$ becomes $\{0\}$,
and the second-order condition \eqref{strsec} of Assumption \ref{ass-interior-3}
trivially holds. Furthermore, the  feasibility constraints of problem \eqref{minsol-3a} have unique solution $\eta=0$, and  thus  $\tilde{\eta}(\cdot)\equiv 0$.
Consequently, the right-hand side of \eqref{optas-2} degenerates into identical zero and the asymptotics of $\hat{\gamma}_N$ are qualitatively different.

To be more specific,
consider $\Xi_N:=\arg\max_{\xi\in \Xi} \hat{f}_{N}(\gamma^*,\xi)$. We have
\[
 \hat{\varphi}_N'(\gamma^*,h)=\max_{\xi\in \Xi_N}h^\top\nabla_\gamma \hat{f}_N(\gamma^*,\xi).
\]
For  $\nabla \hat{\varphi}_{iN}(\gamma)$, $i=1,\ldots,k$,
sufficiently close to their true counterparts,
the respective Lagrange multipliers of problem \eqref{equiv-2} remain positive and thus $\Xi_N=\Xi^*$. Consequently since $\varphi'(\gamma^*,h)>0$ for all $h\ne 0$, it follows that
\begin{equation}\label{empdir-2}
 \hat{\varphi}_N'(\gamma^*,h) >0,\;\forall h\in \bbr^n\setminus\{0\}.
\end{equation}
Suppose further that the set $\Gamma$ is convex and for almost every $X$ and all $\xi\in \Xi$ the function $F(X,\gamma,\xi)$ is convex in $\gamma\in \Gamma$.
By convexity of   $\hat{\varphi}_N(\cdot)$ and $\Gamma$,    it follows from \eqref{empdir-2} that $\gamma^*$ is the unique minimizer of the sample problem.
That is, under the above assumptions,
 the event $\{\hat{\gamma}_N=\gamma^*\}$  happens   w.p.1 for $N$  large enough
(that is, for any point  $\w$ of the respective probability space except in a set of measure zero, there is a number $\bar{N}_\w$ such that for any $N\ge \bar{N}_\w$ the event  holds).
\end{remark}

\section{Conclusion}

In this paper, we have considered asymptotics for
optimal values $\hat{\vv}_N$ and optimal solutions $\hat{\gamma}_N$
of general minimax problems.
Existing literature on the limiting distributions of $\hat{\vv}_N$ and $\hat{\gamma}_N$ is extremely scarce.
We took an approach different from the previous literature that is based on sensitivity analysis of parameterized mathematical optimization problems, and
showed in Sections 2 and 3 that these limiting distributions are in general highly non-Gaussian but reduce to normal distributions in simple cases.
Although our results are somewhat abstract, they open up the possibility of performing statistical inference in minimax problems.

\setstretch{1.05}

\bibliographystyle{elsarticle-harv} 
\bibliography{minimax}

\begin{thebibliography}{25}
\expandafter\ifx\csname natexlab\endcsname\relax\def\natexlab#1{#1}\fi
\providecommand{\url}[1]{\texttt{#1}}
\providecommand{\href}[2]{#2}
\providecommand{\path}[1]{#1}
\providecommand{\DOIprefix}{doi:}
\providecommand{\ArXivprefix}{arXiv:}
\providecommand{\URLprefix}{URL: }
\providecommand{\Pubmedprefix}{pmid:}
\providecommand{\doi}[1]{\href{http://dx.doi.org/#1}{\path{#1}}}
\providecommand{\Pubmed}[1]{\href{pmid:#1}{\path{#1}}}
\providecommand{\bibinfo}[2]{#2}
\ifx\xfnm\relax \def\xfnm[#1]{\unskip,\space#1}\fi
\bibitem[{Bertsekas et~al.(2003)Bertsekas, Nedi{\'{c}} and
  Ozdaglar}]{bertsekas2003convex}
\bibinfo{author}{Bertsekas, D.P.}, \bibinfo{author}{Nedi{\'{c}}, A.},
  \bibinfo{author}{Ozdaglar, A.E.}, \bibinfo{year}{2003}.
\newblock \bibinfo{title}{Convex Analysis and Optimization}.
\newblock \bibinfo{publisher}{Athena Scientific}, \bibinfo{address}{Belmont}.
\bibitem[{Biau et~al.(2020)Biau, Cadre, Sangnier and Tanielian}]{Biau2020AS}
\bibinfo{author}{Biau, G.}, \bibinfo{author}{Cadre, B.},
  \bibinfo{author}{Sangnier, M.}, \bibinfo{author}{Tanielian, U.},
  \bibinfo{year}{2020}.
\newblock \bibinfo{title}{Some theoretical properties of {GAN}s}.
\newblock \bibinfo{journal}{Annals of Statistics} \bibinfo{volume}{48},
  \bibinfo{pages}{1539--1566}.
\bibitem[{Bonnans and Shapiro(2000)}]{BS}
\bibinfo{author}{Bonnans, J.F.}, \bibinfo{author}{Shapiro, A.},
  \bibinfo{year}{2000}.
\newblock \bibinfo{title}{Perturbation Analysis of Optimization Problems}.
\newblock \bibinfo{publisher}{Springer}, \bibinfo{address}{New York}.
\bibitem[{Danskin(1967)}]{danskin1967theory}
\bibinfo{author}{Danskin, J.M.}, \bibinfo{year}{1967}.
\newblock \bibinfo{title}{The Theory of Max-Min and its Application to Weapons
  Allocation Problems}.
\newblock \bibinfo{publisher}{Springer}, \bibinfo{address}{New York}.
\bibitem[{Daskalakis(2018)}]{daskalakis2018equilibria}
\bibinfo{author}{Daskalakis, C.}, \bibinfo{year}{2018}.
\newblock \bibinfo{title}{{Equilibria, Fixed Points, and Computational
  Complexity -- Nevanlinna Prize Lecture}}, in: \bibinfo{booktitle}{Proceedings
  of the International Congress of Mathematicians 2018},
  \bibinfo{organization}{World Scientific}. pp. \bibinfo{pages}{147--209}.
\bibitem[{Delage and Ye(2010)}]{delageye2010}
\bibinfo{author}{Delage, E.}, \bibinfo{author}{Ye, Y.}, \bibinfo{year}{2010}.
\newblock \bibinfo{title}{Distributionally robust optimization under moment
  uncertainty with application to data-driven problems}.
\newblock \bibinfo{journal}{Operations Research} \bibinfo{volume}{58},
  \bibinfo{pages}{595--612}.
\bibitem[{Dem’yanov and Malozemov(1974)}]{demyanov1974introduction}
\bibinfo{author}{Dem’yanov, V.F.}, \bibinfo{author}{Malozemov, V.N.},
  \bibinfo{year}{1974}.
\newblock \bibinfo{title}{Introduction to Minimax}.
\newblock \bibinfo{publisher}{Wiley}, \bibinfo{address}{New York}.
\bibitem[{Diakonikolas et~al.(2021)Diakonikolas, Daskalakis and
  Jordan}]{diakonikolas2021efficient}
\bibinfo{author}{Diakonikolas, J.}, \bibinfo{author}{Daskalakis, C.},
  \bibinfo{author}{Jordan, M.I.}, \bibinfo{year}{2021}.
\newblock \bibinfo{title}{Efficient methods for structured nonconvex-nonconcave
  min-max optimization}, in: \bibinfo{booktitle}{Artificial Intelligence and
  Statistics}, pp. \bibinfo{pages}{2746--2754}.
\bibitem[{Fiacco(1983)}]{Fiacco1983}
\bibinfo{author}{Fiacco, A.V.}, \bibinfo{year}{1983}.
\newblock \bibinfo{title}{Introduction to Sensitivity and Stability Analysis in
  Nonlinear Programming}.
\newblock \bibinfo{publisher}{Academic Press}, \bibinfo{address}{New York}.
\bibitem[{Hettich and Kortanek(1993)}]{HK1993}
\bibinfo{author}{Hettich, R.}, \bibinfo{author}{Kortanek, K.O.},
  \bibinfo{year}{1993}.
\newblock \bibinfo{title}{Semi-infinite programming: Theory, methods, and
  applications}.
\newblock \bibinfo{journal}{SIAM Review} \bibinfo{volume}{35},
  \bibinfo{pages}{380--429}.
\bibitem[{Jin et~al.(2020)Jin, Netrapalli and Jordan}]{jin2020what}
\bibinfo{author}{Jin, C.}, \bibinfo{author}{Netrapalli, P.},
  \bibinfo{author}{Jordan, M.I.}, \bibinfo{year}{2020}.
\newblock \bibinfo{title}{What is local optimality in nonconvex-nonconcave
  minimax optimization?}, in: \bibinfo{booktitle}{International Conference on
  Machine Learning}, pp. \bibinfo{pages}{4880--4889}.
\bibitem[{King and Rockafellar(1993)}]{KR1993}
\bibinfo{author}{King, A.J.}, \bibinfo{author}{Rockafellar, R.T.},
  \bibinfo{year}{1993}.
\newblock \bibinfo{title}{Asymptotic theory for solutions in statistical
  estimation and stochastic programming}.
\newblock \bibinfo{journal}{Mathematics of Operations Research}
  \bibinfo{volume}{18}, \bibinfo{pages}{148--162}.
\bibitem[{Meitz(2024)}]{Meitz2024SJS}
\bibinfo{author}{Meitz, M.}, \bibinfo{year}{2024}.
\newblock \bibinfo{title}{Statistical inference for generative adversarial
  networks and other minimax problems}.
\newblock \bibinfo{journal}{Scandinavian Journal of Statistics}
  \bibinfo{volume}{51}, \bibinfo{pages}{1323--1356}.
\bibitem[{Myerson(1991)}]{myerson1991game}
\bibinfo{author}{Myerson, R.B.}, \bibinfo{year}{1991}.
\newblock \bibinfo{title}{Game Theory: Analysis of Conflict}.
\newblock \bibinfo{publisher}{Harvard University Press},
  \bibinfo{address}{Cambridge}.
\bibitem[{Newey and Smith(2004)}]{neweysmith2004}
\bibinfo{author}{Newey, W.K.}, \bibinfo{author}{Smith, R.J.},
  \bibinfo{year}{2004}.
\newblock \bibinfo{title}{Higher order properties of {GMM} and generalized
  empirical likelihood estimators}.
\newblock \bibinfo{journal}{Econometrica} \bibinfo{volume}{72},
  \bibinfo{pages}{219--255}.
\bibitem[{Owen(2001)}]{owen2001}
\bibinfo{author}{Owen, A.B.}, \bibinfo{year}{2001}.
\newblock \bibinfo{title}{Empirical Likelihood}.
\newblock \bibinfo{publisher}{Chapman and Hall/CRC}, \bibinfo{address}{Boca
  Raton}.
\bibitem[{Shapiro(1985)}]{sha1985}
\bibinfo{author}{Shapiro, A.}, \bibinfo{year}{1985}.
\newblock \bibinfo{title}{Second-order derivatives of extremal-value functions
  and optimality conditions for semi-infinite programs}.
\newblock \bibinfo{journal}{Mathematics of Operations Research}
  \bibinfo{volume}{10}, \bibinfo{pages}{207--219}.
\bibitem[{Shapiro(1989)}]{shap1989}
\bibinfo{author}{Shapiro, A.}, \bibinfo{year}{1989}.
\newblock \bibinfo{title}{Asymptotic properties of statistical estimators in
  stochastic programming}.
\newblock \bibinfo{journal}{Annals of Statistics} \bibinfo{volume}{17},
  \bibinfo{pages}{841--858}.
\bibitem[{Shapiro(1991)}]{shap1991}
\bibinfo{author}{Shapiro, A.}, \bibinfo{year}{1991}.
\newblock \bibinfo{title}{Asymptotic analysis of stochastic programs}.
\newblock \bibinfo{journal}{Annals of Operations Research}
  \bibinfo{volume}{30}, \bibinfo{pages}{169--186}.
\bibitem[{Shapiro(1993)}]{shap1993}
\bibinfo{author}{Shapiro, A.}, \bibinfo{year}{1993}.
\newblock \bibinfo{title}{Asymptotic behavior of optimal solutions in
  stochastic programming}.
\newblock \bibinfo{journal}{Mathematics of Operations Research}
  \bibinfo{volume}{18}, \bibinfo{pages}{829--845}.
\bibitem[{Shapiro(2008)}]{sha08}
\bibinfo{author}{Shapiro, A.}, \bibinfo{year}{2008}.
\newblock \bibinfo{title}{Asymptotics of minimax stochastic programs}.
\newblock \bibinfo{journal}{Statistics and Probability Letters}
  \bibinfo{volume}{78}, \bibinfo{pages}{150--157}.
\bibitem[{Shapiro et~al.(2009)Shapiro, Dentcheva and Ruszczy\'nski}]{SDR}
\bibinfo{author}{Shapiro, A.}, \bibinfo{author}{Dentcheva, D.},
  \bibinfo{author}{Ruszczy\'nski, A.}, \bibinfo{year}{2009}.
\newblock \bibinfo{title}{{Lectures on Stochastic Programming: Modeling and
  Theory}}.
\newblock \bibinfo{publisher}{SIAM}, \bibinfo{address}{Philadelphia}.
\bibitem[{{van der Vaart}(1998)}]{vaart}
\bibinfo{author}{{van der Vaart}, A.W.}, \bibinfo{year}{1998}.
\newblock \bibinfo{title}{Asymptotic Statistics}.
\newblock \bibinfo{publisher}{Cambridge University Press},
  \bibinfo{address}{Cambridge}.
\bibitem[{{von Neumann}(1928)}]{vonneumann1928theorie}
\bibinfo{author}{{von Neumann}, J.}, \bibinfo{year}{1928}.
\newblock \bibinfo{title}{Zur theorie der gesellschaftsspiele}.
\newblock \bibinfo{journal}{Mathematische Annalen} \bibinfo{volume}{100},
  \bibinfo{pages}{295--320}.
\bibitem[{{von Neumann} and Morgenstern(1944)}]{vonneumann1944theory}
\bibinfo{author}{{von Neumann}, J.}, \bibinfo{author}{Morgenstern, O.},
  \bibinfo{year}{1944}.
\newblock \bibinfo{title}{Theory of Games and Economic Behavior}.
\newblock \bibinfo{publisher}{Princeton University Press},
  \bibinfo{address}{Princeton}.

\end{thebibliography}

\end{document}